\documentclass[letterpaper,11pt,leqno]{amsart} 
\usepackage[T1]{fontenc}
\usepackage[T1]{fontenc}
\usepackage[utf8]{inputenc}
\usepackage{bera}
\usepackage[portrait,margin=3cm]{geometry} 
\usepackage[mathscr]{euscript}
\usepackage[colorlinks=true,linkcolor=blue,citecolor=blue,urlcolor=blue]{hyperref} 
\usepackage{amsmath,amssymb,amsthm,amsfonts,amsbsy,latexsym,dsfont,color,graphicx,enumitem}
\usepackage[numeric,initials,nobysame]{amsrefs}

\usepackage{enumitem}
\usepackage{color,comment}

\newcommand{\R}{\mathbb{R}}

\newcommand{\E}{\mathbb{E}}
\renewcommand{\P}{\mathbb{P}}

\newcommand{\T}{\mathbb{T}}

\newcommand{\X}{\mathcal{X}}
\newcommand{\cC}{\mathcal{C}}
\newcommand{\SC}{\mathcal{SC}}

\newcommand{\mpd}{m_{p, d}(\mathcal{X})}
\newcommand{\Mpd}{M_{p, d}(\mathcal{X})}
\newcommand{\M}{\mathcal{M}}

\newtheorem{thm}{Theorem}[section]
\newtheorem{corollary}[thm]{Corollary}

\newtheorem{lemma}[thm]{Lemma}

\newtheorem{remark}[thm]{Remark}

\newtheorem{prop}[thm]{Proposition}
\newtheorem{conjecture}[thm]{Conjecture}

\newtheorem{theorem}[thm]{Theorem}
\theoremstyle{definition}
\newtheorem{df}[thm]{Definition}
\newtheorem*{claim*}{Claim}

\numberwithin{equation}{section}
\newtheorem*{acknowledgement}{Acknowledgement}

\begin{document}

\title[]{Bounds for exit times of Brownian motion and the first Dirichlet eigenvalue for the Laplacian}
%\textcolor{red}{(Need shorter title! By the way, we never really use "bottom of spectrum")}
\author[Banuelos]{Rodrigo Ba\~{n}uelos{$^{\dag}$}}
\thanks{\footnotemark {$\dag$} Research was supported in part by NSF Grant DMS-1854709}
\address{ Department of Mathematics\\
Purdue University\\
West Lafayette, IN 47907,  U.S.A.} \email{banuelos@purdue.edu}

\author[Mariano]{Phanuel Mariano{$^\star$}}
\thanks{\footnotemark {$\star$} Research was supported in part by an AMS-Simons Travel Grant 2019-2023.}
	\address{Department of Mathematics\\
		Union College\\
		Schenectady, NY 12308,  U.S.A.}
	\email{marianop@union.edu}

\author[Wang]{Jing Wang{$^{\ddag}$}}
\thanks{\footnotemark {$\ddag$} Research was supported in part by NSF Grant DMS-1855523}
\address{ Department of Mathematics\\
Purdue University\\
West Lafayette, IN 47907,  U.S.A.}
\email{jingwang@purdue.edu}


%\begin{document}

\begin{abstract}  For domains in $\mathbb{R}^d$, $d\geq 2$, we prove  universal upper  and lower bounds on the product of the bottom of the spectrum
for the Laplacian to the power $p>0$ and the supremum over all starting points of the $p$-moments of the exit time of Brownian motion.  It is shown that the lower bound is sharp for integer values of $p$ and that for $p \geq 1$, the upper bound is asymptotically sharp as $d\to\infty$. For all $p>0$,  we prove the existence of an extremal domain among the class of domains that are convex and symmetric with respect to all coordinate axes.  For this class of domains we conjecture that the cube is extremal.  

\end{abstract}

\maketitle

\tableofcontents

\section{Introduction and statements of main results}

There is a large class of results often referred to as generalized isoperimetric inequalities that have wide  interest in both the mathematics and physics community, see Poly\'a and Szego \cite{Polya-Szego-1951} and Bandle \cite{Bandle-1980}. At the heart of these inequalities is the classical isoperimetric inequality which states that among all regions of fixed volume, surface area is minimized by  balls.  In spectral theory among the classical results is the celebrated Rayleigh-Faber-Krahn inequality which states that among all domains $D\subset \R^d$ having the same volume as a ball $D^{*}$,  
\begin{equation}\label{egen1} 
\lambda_1(D) \geq \lambda_1(D^*),
\end{equation}
 where $\lambda_1(D)$ denotes the first Dirichlet eigenvalue for the Laplacian in $D$.  Further, equality holds if and only if $D$ is a ball. Without loss of generality we take $D^*$ to be  centered at the origin.

On the other hand, it has also been known for many years that one can state many of these inequalities in terms of the exit time of Brownian motion from the domain $D$.  This probabilistic connection provides new insights and  raises new interesting questions on their validity for processes other than Brownian motion, such as L\'evy processes.  To illustrate, let  $B_{t}$ be a $d-$dimensional Brownian motion starting at the point $x\in D$ 
and let $\tau_{D}=\inf\left\{ t>0\mid B_{t}\notin D\right\}$ 
be its first exit time from $D$. Using the symmetrization techniques for multiple integrals in \cite{Brascamp-Lieb-Luttinger-1974,Luttinger-1973a,Luttinger-1973b} it follows that 

\begin{equation}\label{Lifetime2}
\sup_{x\in D}\mathbb{P}_{x}\left(\tau_{D}>t\right)\leq\mathbb{P}_{0}\left(\tau_{D*}>t\right),
\end{equation}
for all $t>0$.  In particular, for any $p>0$, 

\begin{equation}\label{Lifetime}
\sup_{x\in D}\mathbb{E}_{x}\left[\tau_{D}^p\right]\leq\mathbb{E}_{0}\left[\tau_{D^*}^p\right].
\end{equation}
Equality holds in these inequalities if and only if $D$ is a ball.  Inequality \eqref{egen1} follows from inequality \eqref{Lifetime2} by taking into account the classical result that for any $D\subset \R^d$, 
$$
\lim_{t\to\infty}\frac{1}{t} \log\mathbb{P}_{x}\left(\tau_{D}>t\right)=-\frac{\lambda_1(D)}{2}.$$
In a similar way, the classical isoperimetric inequality can be obtained from isoperimetric inequalities for exit times of Brownian motion using small time behavior. These are now classical results with many extensions and applications  that can be found in
\cite{Aizenman-Simon-1982, Burchard-etal2001, Talenti1976, Banuelos2010a} and many other references given in these papers. 

From the connections among \eqref{egen1}, \eqref{Lifetime2} and \eqref{Lifetime} we can already observe a competing relation between  $\lambda_1(D)$ and $\sup_{x\in D}\mathbb{E}_{x}\left[\tau_{D}\right]$. Indeed fine connections between the two quantities have been investigated for many years by many authors.  Consider the domain functional 
$
G\left(D\right)  =\lambda_{1}\left(D\right)\sup_{x\in D}\mathbb{E}_{x}\left[\tau_{D}\right].
$
It was proven in \cite{Banuelos-Carrol1994} that for all simply connected domains $D$ in $\R^2$, 
\begin{equation}\label{ExpEig-BanCar}
2\leq G(D)\leq\frac{7\zeta\left(3\right)j_{0}^{2}}{8}\approx 6.08.
\end{equation}
In higher dimensional spaces, it is easy to show the lower bound $G\left(D\right)\ge 2$ for all domains $D\subset \R^d$ (see \cite{Banuelos-Carrol1994} and Section \ref{sec:LowerBound} below). In \cite{Vandenberg-2017a, Henrot-Lucardesi-Philippin-2018} the authors independently show that $2$  is in fact a sharp lower bound for all bounded domains. Many results have been devoted to obtaining the upper bound estimates for $G(D)$ (see  \cite{Vandenberg-Carroll-2009}, \cite{Giorgi-Smits-2010} and  \cite{Vogt-2019a}). In particular the recent paper \cite{Vogt-2019a} improves the upper bound for all general domains to 
\begin{equation}\label{VogtBound}
\frac{d}{4}+\frac{\sqrt{d}}{2}\sqrt{5\left(1+\frac{1}{4}\log2\right)}+2.
\end{equation}
With the leading term $\frac{d}{4}$, this bound is asymptotically sharp as $d\to\infty$.
However the question of proving a sharp upper bound is wide open even when restricted to special classes of domains such as planar simply connected or convex domains.

The main object of study of this paper is  the shape functional
\begin{equation}\label{eq-shape-func}
G_{p,d}\left(D\right)  =\lambda_{1}^{p}\left(D\right)\sup_{x\in D}\mathbb{E}_{x}\left[\tau_{D}^{p}\right]. 
\end{equation} 
Here $\lambda_{1}\left(D\right)$ is the first Dirichlet eigenvalue
for the Laplacian in $D$. When the Laplacian has no discrete spectrum
in $D$, then we take $\lambda_{1}\left(D\right)$ to be the \textbf{bottom
of the spectrum} for the Laplacian which is given by $\lambda_{1}\left(D\right)=\inf_{\phi}\frac{\int_{D}\left|\nabla\phi\right|^{2}dx}{\int_{D}\phi^{2}dx}$
where $\phi\in H_{0}^{1}\left(D\right), \phi\neq 0$.  The $p$-moments of the exit time $\mathbb{E}_{x}\left[\tau_{D}^{p}\right]$ can also be stated in terms of the {torsion function}. For $0<p<\infty$, the  $p-${\it torsion moment function} 
 $u_{p}:D\to\bar{\mathbb{R}}_{+}$ is defined by
\begin{equation}\label{p-torsion} 
u_{p}(x)=\frac{1}{2^{p}\Gamma\left(p+1\right)}\mathbb{E}_{x}\left[\tau_{D}^{p}\right].
\end{equation} 
Not to be confused  with the $p-$torsion function related to the $p$-Laplacian \cite{Brasco-2017,Della-etall-2018,Giorgi-Smits-2010}. When $k\in\mathbb{N}$, and $\sup_{x\in D}u_{k}(x)<\infty$, these functions are solutions to 
\begin{equation}\label{p-Torsion-PDE}
\begin{cases}
-\Delta u_{1}=1 & u_{1}\in H_{0}^{1}\left(D\right)\\
-\Delta u_{k}=u_{k-1} & u_{k}\in H_{0}^{1}\left(D\right),k=2,3,\dots
\end{cases}
\end{equation}
When $p=1$, $u_{1}(x)=\frac{1}{2}\E_x[\tau_D]$ is the classical torsion function which has been extensively studied in the literature with applications to many areas of mathematics and mathematical physics. See for instance the classical works \cite{Bandle-1980,Landau-Lifshitz-1959,Polya-Szego-1951}.  

For general $p$, the literature in the study of exit time moments %as well as their maximum, their finiteness 
and their applications to different fields is extremely large by now. In the works \cite{Boudabra-Markowsky-2020a,Boudabra-Markowsky-2021a,Boudabra-Markowsky-2019,Markowsky-2015} Boudabra-Markowsky studied exit time moments, proved results on the location of their maximum and gave conditions for their finiteness. The $k$-torsion moment functions $u_{k}$ have also been applied to the study of heat flow by McDonald-Meyers-Meyerson in \cite{McDonald-2013,McDonald-Meyers-2003,Meyerson-McDonald-2017}. In \cite{laPena-McDondald-2004}, de la Pe\~na-McDonald provide an algorithm that produces uniform approximations of arbitrary continuous functions by exit time moments. We also point to the work of \cite{Hurtado-Markvorsen-Palmer-2016}, where Hurtado-Markvorsen-Palmer use the $L^1$ norms of $u_k$ to give an alternative characterization for $\lambda_1 (D)$ on Riemannian manifolds. In the closely related papers \cite{Colladay-Langford-McDonald-2018,Dryden-Langford-McDonald-2017}, the authors  give upper bounds on $\lambda_1 (D)$ using the $L^1$ norms of exit time moments on manifolds. Moreover, obtaining precise spectral bounds for the exit time in large dimensions has been of interest in other settings. For example in \cite{Panzo2020}, Panzo obtains a spectral bound for the torsion function of symmetric stable processes that has the correct order of growth. For other applications to the study of exit time moments, torsional rigidity, stability, the study of minimal sub-manifolds, and optimal trapping of Brownian motion and gradient estimates, we refer the reader to \cite{Mariano-Panzo-2020,Hoskins-Steinerberger-2019,Kim-2019,Lu-Steinerberger2019b,Markvorsen-Palmer-2006}. \\

%{\color{red} Jing: maybe we can shorten a bit here?}

Our main goals  in this paper are to investigate sharp  bounds for $G_{p,d}(D)$, $D\in \mathcal{X}$ where $\mathcal{X}$ contains all  domains in $\R^d$ such that $\lambda_1(D)>0$;  and prove the existence of their extremals in the class of convex domains which are symmetric with respect to each coordinate axis. 

For the rest of this paper we will work with the  function $\mathbb{E}_{x}\left[\tau_{D}^{p}\right]$ and leave the trivial translation of the bounds for  $u_{p}(x)$ to the interested reader.  For a given class of domains $\mathcal{D}$, define 
\begin{equation}\label{maxEigenExpectred} 
M_{p, d}(\mathcal{D})=\sup_{D\in \mathcal{D}}G_{p, d}(D),
\end{equation}
and
\begin{equation}\label{minEigenExpectred}
m_{p, d}(\mathcal{D})=\inf_{D\in \mathcal{D}}G_{p, d}(D).
\end{equation}

Our first main result provides a sharp asymptotic upper bound for $\Mpd$ and a sharp lower for $\mpd$. 

\begin{theorem}[Sharp Lower and Asymptotic Upper Bounds in $\X$]\label{asymtotic} For $p\geq 1$, 
\begin{equation}\label{asymtotic-bound}
\lim_{d\to\infty} \frac{M_{p, d}(\X)}{d^p}=\frac{1}{4^{p}}.
\end{equation}

Moreover, if $p>0$ then 
\begin{equation}\label{LowerBound-Exittime-p} 
2^{p}\Gamma\left(p+1\right)\leq \mpd.
\end{equation}
Furthermore, \eqref{LowerBound-Exittime-p} is sharp when $p$ takes values in $\mathbb{N}$.
\end{theorem} 

The upper bound asymptotic \eqref{asymtotic-bound} is accomplished by working on precise universal upper bounds for $G_{p, d}$ (see Section \ref{Sec:3}, more precisely, Theorems \ref{thm-main-thm} and \ref{thm:SharpUpper}). The lower bound \eqref{LowerBound-Exittime-p} has been independently obtained by Biswas-L\H{o}rinczi in \cite{Biswas-Lorinczi}, but only for the restricted class of general convex domains. We prove the lower bound holds for any domain as long as the bottom of spectrum is positive. Moreover, our result proves sharpness for any integer $p$. 
\\

%%%%%%%

%Our first main result provides sharp lower bounds for $\mpd$ where $\X$ contains all  domains in $\R^d$ such that $\lambda_1(D)>0$.
%\begin{theorem}[Sharp Lower Bound in $\X$]\label{THM:Lower}
%Fix $p>0$.  
%Then 
%\begin{equation}\label{LowerBound-Exittime-p} 
%2^{p}\Gamma\left(p+1\right)\leq \mpd.
%\end{equation}
%Furthermore, the inequality is sharp when $p$ takes values in $\mathbb{N}$. 
%\end{theorem}
%
%
%
%
%
%
%Regarding the upper bounds $\Mpd$, by working on precise universal upper bounds on $G_{p, d}$ (see Section \ref{Sec:3}, more precisely, Theorems \ref{thm-main-thm} and \ref{thm:SharpUpper}.) we are able to obtain the following sharp asymptotic estimate.
%%Next, we prove in Section \ref{Sec:4} the following asymptotic upper bound. 
%
%\begin{theorem}[Sharp Asymptotic Upper Bound in $\X$]\label{asymtotic} For $p\geq 1$, 
%\begin{equation}
%\lim_{d\to\infty} \frac{M_{p, d}(\X)}{d^p}=\frac{1}{4^{p}}.
%\end{equation}
%\end{theorem}  
%The sharpness of the above estimate can be easily seen by  comparing with a ball. \\

%%%%%%%%%%%%%%%%%%%%%%%%%%%%%%%%%%%%%%%%%%%%%%%%%%%%%%%%%%%%%%%%%%%%%%%%%

Our second main result concerns the existence of extremals for  classes of domains. Given the isoperimetric inequalities \eqref{egen1} and \eqref{Lifetime} (as well as other inequalities where balls are extremals in $\mathcal{X}$), one could speculate about the maximality of a ball $B$ for these extremal problems.  However it was pointed out in \cite[pg. 599]{Banuelos-Carrol1994} that 
\begin{equation}\label{B-vs-T}
\lambda_1\left(B\right)\sup_{x\in B}\E_{x}\left[\tau_{B}\right]< \lambda_1\left(\T\right)\sup_{x\in \T}\E_{x}\left[\tau_{\T}\right],
\end{equation}
where $\T$ is is the equilateral triangle. 
%This was also later observed in  \cite[Corollary 3.7]{Henrot-Lucardesi-Philippin-2018}.
%It  was conjectured in \cite{Henrot-Lucardesi-Philippin-2018} that no extremal domain exists over
%the class of all domains. 
%% for the torsion quantity $M_{1, d}(T, \lambda)$ and, of course, this would be the same for $M_{1, d}(\E, \lambda)$. 
%The authors 
%give evidence for their conjecture in their Proposition 3.1. Their
%techniques involve defining a shape derivative. They show that in
%the class of bounded open $C^{2}$ domains, if $\lambda_{1}\left(D\right)\left\Vert u_{D}\right\Vert _{\infty}$ is differentiable
%at $D$, then $D$ is not a maximizer.
%
%
%
%From the above discussion it is reasonable, when looking for extremals, to restrict the class of domains. 
%In \cite{Payne-1981}, Payne showed that
%\begin{equation}\label{PayneConvex}
%m_{1, d}(\E, \lambda)=\frac{\pi^{2}}{4},
%\end{equation}
%where the infimum in  \eqref{minEigenExpectred} is now taken over all convex domains.  
%From this it follows trivially that the minimizer domain over convex domains is given by the infinite slab $S_d=\mathbb{R}^{d-1}\times\left(-1,1\right)$. 
The existence of maximizers in the class of convex domains is proved in  \cite{Henrot-Lucardesi-Philippin-2018}. In the same paper the authors conjecture that when 
$d=2$, the equilateral triangle $\T$ is an extremal for $M_{1,2}(\X)$. 

In this paper we are interested in the extremals, particularly their existence, for the shape functional $G_{p,d}$ among the class of bounded convex domains that are \textbf{doubly symmetric} (symmetric with respect to the both coordinate axes). 
There have been many interesting problems concerning the geometry of the Laplacian in such domains and substantial progress  has been made. We refer the readers to some of this large literature \cite{Banuelos-Burdzy-1999, Banuelos-Pang-Pascu-2004,Pascu-2002,Jerison-Nadirashvili-2000, Banuelos-Kulczycki2006,Banuelos-Mendez-Hernandez-2000,Davis-2001, Andrews-Clutterbuck-2011}.

\begin{df}
Let $\mathcal{C}$ be the class of bounded convex domains
in $\mathbb{R}^{d}$, $d\geq 2$. Let $\mathcal{SC}$ be the subclass of domains in $\mathcal{C}$ 
%\textbf{open
%bounded convex sets} 
that are symmetric
with respect to each coordinate axis.
\end{df}

%Furthermore, consider the following shape functional, $$G_{p}\left(D\right)  =\lambda_{1}^{p}\left(D\right)\sup_{x\in D}\mathbb{E}_{x}\left[\tau_{D}^{p}\right]. $$ With this notation, 
%$$
%M_{p, d}(\E, \lambda)=\sup_{D}G_{p}\left(D\right).$$

We obtain the following result. 

\begin{thm}[Existence of extremals in $\mathcal{C}$ or $\mathcal{SC}$]\label{Thm:Existence}
For any $p>0$, $d\ge2$, the upper bounds $M_{p,d}(\cC)$ and $M_{p,d}(\SC)$ admit extremals. 
\end{thm}

\begin{remark} The case for $p=1$ and the class $\mathcal{C}$ is proved in \cite{Henrot-Lucardesi-Philippin-2018}. Our  proof of Theorem \ref{Thm:Existence} depends  on a key estimate (Lemma \ref{Lem:Existence1a}) which estimates the $p-$moment of the difference $(\tau_{D}-\tau_{U})$, where $U\subset D\subset \R^d$ and $D$ is a bounded Lipschitz domain.   
 This in turn will allow us to show in Proposition \ref{prop:Existence2} that $G_{p, d}$ is continuous with respect to the Hausdorff distance. This is quite different from the proof in the special case in \cite{Henrot-Lucardesi-Philippin-2018,Henrot-Pierre-2018} which uses purely PDE techniques.   
\end{remark}

%\begin{enumerate}
%\item \textbf{Sharp Lower Bound} (Theorem \ref{THM:Lower}): Fix $p>0$. 
%%Let $D$ be domain in $\mathbb{R}^{d}$ satisfying  $\lambda_1(D)>0$. 
%Then 
%\begin{equation*} 
%2^{p}\Gamma\left(p+1\right)\leq m_{p, d}(\E, \lambda).
%\end{equation*}
%Furthermore, the inequality is an equality when $p$ takes values in $\mathbb{N}$. 
%\item \textbf{Sharp Assymptotic Upper Bound} (Theorem \ref{asymtotic}): For $p\geq 1$, 
%\begin{equation*}
%\lim_{d\to\infty} \frac{M_{p, d}(\E, \lambda)}{d^p}=\frac{1}{4^{p}}.
%\end{equation*}
%
%\item \textbf{Existence of Extremals} (Theorem \ref{Thm:Existence}): 
%For any $p>0$, the shape functional 
%$$G_{p}\left(D\right)  =\lambda_{1}^{p}\left(D\right)\sup_{x\in D}\mathbb{E}_{x}\left[\tau_{D}^{p}\right] $$
%admits a maximizer in
%the class $\mathcal{C}$ (of bounded convex domains
%in $\mathbb{R}^{d}$) or in the class $\mathcal{SC}$ (of domains in $\mathcal{C}$ which are symmetric
%with respect to each coordinate axes).
%\end{enumerate}
%
%\textcolor{blue}{or should I simply move the main theorem here.}

%{\bf Summary of Results and Outline:}

%\vspace{10pt}
%The paper is organized as follows. 

The paper is organized as follows. Upper bounds are contained  in Sections \ref{Sec:3},\ref{Sec:4} in Theorems \ref{MAIN-UpperBound-Theorem} and \ref{thm:SharpUpper}. The proof of Theorem \ref{asymtotic} is split into Sections \ref{Sec:4} and \ref{sec:LowerBound}. The proof of the asymptotic upper bound \eqref{asymtotic-bound} is given in  Section \ref{Sec:4}. The proof of the lower bound \eqref{LowerBound-Exittime-p} is given in  Section \ref{sec:LowerBound}. In Section \ref{Sec:Extremal} we discuss the problems of finding extremal domains for $G_{p,d}$  restricted to various subclasses of domains.  The proof of Theorem \ref{Thm:Existence} is given in Sections \ref{Sec:5.3} and \ref{Sec:5.4}.  This section also contains a conjecture on the extremal domain for the class  $\mathcal{SC}$, Conjecture \ref{Conjecture-Double-Symmetric}.

\section{Upper Bounds for $\Mpd$}\label{Sec:3}

In this section we obtain some preliminary  upper bound estimates that will allow us to prove the sharp asymptotic upper bound for $\Mpd$, which will be done in Section \ref{Sec:4}.

Let $K_{D}\left(x,y,t\right)$ be the {\it Dirichlet heat kernel} for $\Delta_D$ in the domain $D$.  The transition density $p_{D}$  for Brownian
motion killed upon leaving  $D$ is given by
\[
p_{D}\left(x,y,t\right)=K_{D}\left(x,y,t/2\right),
\]
as $\frac{1}{2} \Delta_D$ is the generator of Brownian motion. We can then write 
\begin{align*}
\mathbb{E}\left[\tau_{D}^{p}\right] & =p\int_{0}^{\infty}t^{p-1}\mathbb{P}_{x}\left(\tau_{D}>t\right)dt=p\int_{0}^{\infty}\int_{D}t^{p-1}p_{D}\left(x,y,t\right)dydt\\
 & =2^{p}p\int_{0}^{\infty}\int_{D}s^{p-1}K_{D}\left(x,y,s\right)dyds.
\end{align*}
We also recall the classical  upper incomplete gamma function
\begin{align*}
%\gamma\left(s,x\right) & =\int_{0}^{x}u^{s-1}e^{-u}%du,\\
\Gamma\left(s,x\right) & =\int_{x}^{\infty}u^{s-1}e^{-u}du.
\end{align*}
 
\begin{thm}\label{thm-main-thm}
\label{MAIN-UpperBound-Theorem}
%\textcolor{green}{I don't see where $p\geq 1$ is needed in the proof. Can you confirm theorem can be restate with $p\geq 0$?}
For any $p>0$, we have 
%\begin{equation}
%M_{p, d}(T, \lambda)\leq C_{1}(d, p), \label{eq:VOGTImpovement1b}
%\end{equation} 
%and 
\begin{equation}
\Mpd\leq 2^p\Gamma(p+1)C_{1}(d,p),\label{eq:VOGTImpovement1}
\end{equation}
where
\begin{align}
 & C_{1}\left(d,p\right):=\nonumber\\
 & \inf_{a>0,0<\epsilon<1}\left\{\frac{a^{p}}{2^{p}\Gamma\left(p+1\right)}+\frac{1}{\Gamma\left(p\right)}\frac{e^{d/4}\sqrt{2}}{\left(8d\right)^{d/4}}\sqrt{\frac{\Gamma\left(d\right)}{\Gamma\left(d/2\right)}}\left(1+\frac{1}{\sqrt{\epsilon}}\right)^{d/2}\frac{\Gamma\left(p,\left(1-\epsilon\right)a/2\right)}{\left(1-\epsilon\right)^{p}}\right\}. \label{eq:C2}
\end{align}
\end{thm}

\begin{proof} 
For any $D\in \X$, since  $\mathbb{E}\left[\tau_{D}^{p}\right]=\int_0^\infty t^{p-1}\P(\tau_D>t)dt$. We consider splitting the integral at the bottom of the spectrum $\lambda_1$ of the Dirichlet Laplacian.
Precisely, for any $x\in D$ and $a>0$ we have
\begin{align}\label{eq-E-tau-p}
\mathbb{E}_x\left[\tau_{D}^{p}\right]
%  &=p\int_{0}^{\infty}t^{p-1}\mathbb{P}_{x}\left(\tau_{D}>t\right)dt\\
 & =p\int_{0}^{a/\lambda_{1}}t^{p-1}\mathbb{P}_{x}\left(\tau_{D}>t\right)dt+p\int_{a/\lambda_{1}}^{\infty}t^{p-1}\mathbb{P}_{x}\left(\tau_{D}>t\right)dt\nonumber \\
 & \leq \frac{a^{p}}{\lambda_1^p}+p\int_{a/\lambda_1}^{\infty}t^{p-1}\mathbb{P}_{x}\left(\tau_{D}>t\right)dt
 %&= \frac{a^{p}}{\lambda_1^p}+pI,
\end{align}
Let
$
I=\int_{a/\lambda_{1}}^{\infty}t^{p-1}\mathbb{P}_{x}\left(\tau_{D}>t\right)dt.
$
The theorem is proved upon obtaining the estimate for $I$ that we give in the next lemma.
\end{proof}

\begin{lemma}\label{lemma-EstimateIII}
For any $x\in D$, $a>0$, we have 
\begin{equation}\label{eq-II-delta}
I\le 2^{p}\frac{e^{d/4}\sqrt{2}}{\left(8d\right)^{d/4}}\sqrt{\frac{\Gamma\left(d\right)}{\Gamma\left(d/2\right)}}\left(1+\frac{1}{\sqrt{\epsilon}}\right)^{d/2}\frac{\Gamma\left(p,\left(1-\epsilon\right)a/2\right)}{\left(1-\epsilon\right)^{p}\lambda_{1}^{p}}.
\end{equation}
\end{lemma}
%Plugging $(\ref{FinalI}),(\ref{eq:FinalII}),(\ref{eq:FinalIII})$ we get 
The proof of the above lemma relies on some improvement of  Vogt's result in \cite{Vogt-2019a}. We split  the major steps into the lemma and proposition below. %consists several parts.  First we will need the following improvement estimate on Vogt's result in.  
%to obtain our upper bound for I . %The results
%in \cite{Vogt-2019a} use very clever estimates to bound the $L^{p}\to L^{q}$
%norms of certain operators related to the Dirichlet Laplacian.  We make a simple improvement
%to the \cite[Proposition 2.5]{Vogt-2019a}.

\begin{lemma}
\label{Vogt-2.5-Improvement}Let $D\subset\mathbb{R}^{d}$ be measurable,
$\alpha>0,$ and let $L$ be a bounded operator on $L^2(D)$
satisfying 
\begin{equation}
\left\Vert e^{-\alpha\rho_{w}}Le^{\alpha\rho_{w}}\right\Vert _{2\to\infty}\leq1\label{Vogt2.5-1}
\end{equation}
for all $w\in D$, where  $\rho_{w}\left(x\right)=\left|x-w\right|,w\in\mathbb{R}^{d}$. Then 
\begin{equation}
\left\Vert L\right\Vert _{\infty\to\infty}\leq\frac{\sqrt{2}\pi^{d/4}}{\left(2\alpha\right)^{d/2}}\sqrt{\frac{\Gamma\left(d\right)}{\Gamma\left(d/2\right)}}\label{Vogt2.5-2}
\end{equation}
\end{lemma}

\begin{proof}
The proof is essentially the same as  in \cite[Proposition 2.5]{Vogt-2019a}. Note that $\|Lf\|_\infty=\sup_{w\in D}\|e^{-\alpha\rho_w}Lf\|_\infty$. Then we have
\[
\|e^{-\alpha\rho_w}Lf\|_\infty\le \|e^{-\alpha\rho_w}f\|_2\le \|e^{-\alpha\rho_w}\|_2\|f\|_\infty.
\]
Let $\sigma_{d-1}$ denote the
surface measure of the unit sphere,  
then the conclusion follows from the estimate below.
\begin{align*}
\left\Vert e^{-\alpha\rho_{w}}\right\Vert _{2}^{2} & =\int e^{-2\alpha\left|w-y\right|}\chi_{D}(y)dy
  \leq\int e^{-2\alpha\left|y\right|}dy\\
 & =\sigma_{d-1}\int_{0}^{\infty}e^{-2\alpha r}r^{d-1}dr
 % =\sigma_{d-1}\frac{1}{\left(2\alpha\right)^{d}}\int_{0}^{\infty}e^{-u}u^{d-1}du\\
 %& =\sigma_{d-1}\frac{\Gamma\left(d\right)}{\left(2\alpha\right)^{d}}
 =\frac{2\pi^{d/2}}{\Gamma\left(d/2\right)}\frac{\Gamma\left(d\right)}{\left(2\alpha\right)^{d}}.
\end{align*}
\end{proof}
 
%As a consequence we then are able to obtain the following estimates that improve the result in \cite[Theorem 2.1]{Vogt-2019a}. %Keeping track of the estimates in the proof of Theorem 2.1 in \cite{Vogt-2019a} line by line with the new estimate on $\left\Vert B\right\Vert _{\infty\to\infty}$ in $(\ref{Vogt2.5-2})$ for bounded operators satisfying $(\ref{Vogt2.5-1})$ yields the following bound. 
\begin{prop}
\label{VogThm2.1-Improvement}For all $\epsilon\in\left(0,1\right]$, we have
\[
\left\Vert e^{-t\left(-\Delta_D\right)}\right\Vert _{\infty\to\infty}\leq e^{d/4}\frac{\sqrt{2}}{\left(8d\right)^{d/4}}\sqrt{\frac{\Gamma\left(d\right)}{\Gamma\left(d/2\right)}}\left(1+\frac{1}{\sqrt{\epsilon}}\right)^{d/2}e^{-\left(1-\epsilon\right)\lambda_{1}t},
\]
for $t\geq0$.  In particular, for all $x\in D$ and $t\geq 0$, 
\begin{equation}\label{IIIestiamte}
\P_x(\tau_D>t)\leq e^{d/4}\frac{\sqrt{2}}{\left(8d\right)^{d/4}}\sqrt{\frac{\Gamma\left(d\right)}{\Gamma\left(d/2\right)}}\left(1+\frac{1}{\sqrt{\epsilon}}\right)^{d/2}e^{-\left(1-\epsilon\right)\frac{\lambda_1t}{2}}.
\end{equation}

\end{prop}

\begin{proof}
The proof is similar to Theorem 2.1 in \cite{Vogt-2019a}. Here we only sketch the key steps. Consider the operator $H=-\Delta_D-\lambda_1$, it 
 is a self-adjoint operator in $L^{2}\left(D\right)$ with the bottom of spectrum $\lambda_1(H)=0$. Clearly the heat kernel of $e^{-tH}$ has the Gaussian upper bound
 \[
\left|K_{t}(x,y)\right|\leq e^{\lambda_1 t}\cdot\frac{1}{\left(4\pi t\right)^{d/4}}\exp\left(-\frac{\left|x-y\right|^{2}}{4t}\right)
 \]
 for all $t>0$ and a.e. $x,y\in D$. 
It then holds that (see proof of Theorem 2.1, page 43 in \cite{Vogt-2019a})
%a general operator $H$ satisfying Gaussian upper bounds (see Assumption
%A in \cite{Vogt-2019a}). 
%Using
%the fourth equation given in the proof of Theorem 2.1 in \cite{Vogt-2019a})
%(with $M=1,\omega=0,a=4$) we have 
for any $\epsilon\in\left(0,1\right]$
\begin{align}
\left\Vert e^{-\alpha\rho_{w}}e^{-tH}e^{\alpha\rho_{w}}\right\Vert _{2\to\infty} & \leq\left(8\pi\epsilon t\right)^{-d/4}\left(1+\frac{1}{\beta}\right)^{d/4}e^{\lambda_1\epsilon t}e^{\left(1+\beta\right)\alpha^{2}\epsilon t+\alpha^{2}\left(1-\epsilon\right)t}.\label{eq:VogtB-Bound1}
\end{align}
Applying Lemma \ref{Vogt-2.5-Improvement}
to $L=e^{-tH}$ and using $(\ref{eq:VogtB-Bound1})$ we have that 

\begin{align*}
\left\Vert e^{-tH}\right\Vert _{\infty\to\infty} & \leq\frac{\sqrt{2}\pi^{d/4}}{\left(2\alpha\right)^{d/2}}\sqrt{\frac{\Gamma\left(d\right)}{\Gamma\left(d/2\right)}}\cdot\left(8\pi\epsilon t\right)^{-d/4}\left(1+\frac{1}{\beta}\right)^{d/4}e^{\lambda_1\epsilon t}e^{\left(1+\beta\epsilon\right)\alpha^{2}t},
\end{align*}
taking $\alpha^{2}=\frac{d/4}{\left(1+\beta\epsilon\right)t}$ we
obtain %$\left(2\alpha\right)^{d/2}=\left(\frac{d}{\left(1+\beta\epsilon\right)t}\right)^{d/4}$, hence
\begin{align*}
\left\Vert e^{-tH}\right\Vert _{\infty\to\infty} %& \leq\frac{\sqrt{2}\pi^{d/4}}{\left(\frac{d}{\left(1+\beta\epsilon\right)t}\right)^{d/4}}\sqrt{\frac{\Gamma\left(d\right)}{\Gamma\left(d/2\right)}}\left(8\pi\epsilon t\right)^{-d/4}\left(1+\frac{1}{\beta}\right)^{d/4}e^{\lambda_1\epsilon t+d/4}\\
 & \le\sqrt{2}\pi^{d/4}\sqrt{\frac{\Gamma\left(d\right)}{\Gamma\left(d/2\right)}}\left(8\pi\epsilon t\right)^{-d/4}\left(\frac{\left(1+\beta\epsilon\right)t}{d}\left(1+\frac{1}{\beta}\right)\right)^{d/4}e^{\lambda_1\epsilon t+d/4}.
\end{align*}
Optimizing the right hand side of the above inequality by taking $\beta=\epsilon^{-1/2}$ we have
\begin{align*}
\left\Vert e^{-tH}\right\Vert _{\infty\to\infty} & \leq%\sqrt{2}\pi^{d/4}\sqrt{\frac{\Gamma\left(d\right)}{\Gamma\left(d/2\right)}}\left(8\pi\epsilon t\right)^{-d/4}\left(\frac{t}{d}\left(1+\epsilon^{1/2}\right)^{2}\right)^{d/4}e^{d/4}\\
 %& =\frac{\sqrt{2}\pi^{d/4}}{\left(8\pi d\right)^{d/4}}\sqrt{\frac{\Gamma\left(d\right)}{\Gamma\left(d/2\right)}}\cdot\left(1+\frac{1}{\sqrt{\epsilon}}\right)^{d/2}e^{d/4}\\
 e^{d/4}\frac{\sqrt{2}}{\left(8d\right)^{d/4}}\sqrt{\frac{\Gamma\left(d\right)}{\Gamma\left(d/2\right)}}\cdot\left(1+\frac{1}{\sqrt{\epsilon}}\right)^{d/2}e^{\lambda_1\epsilon t}.
\end{align*}
This then completes the proof.
%Therefore we have 
%\[
%\left\Vert e^{-t\left(-\Delta\right)}\right\Vert _{\infty\to\infty}\leq e^{d/4}\frac{\sqrt{2}}{\left(8d\right)^{d/4}}\sqrt{\frac{\Gamma\left(d\right)}{\Gamma\left(d/2\right)}}\left(1+\frac{1}{\sqrt{\epsilon}}\right)^{d/2}e^{-\left(1-\epsilon\right)\lambda_{1}t}.
%\]
\end{proof}
\begin{proof}[Proof of Lemma \ref{lemma-EstimateIII}]
By \eqref{IIIestiamte} we have 
\begin{align}\label{eq:FinalIII}
I & =\int_{a/\lambda_{1}}^{\infty}t^{p-1}\mathbb{P}_{x}\left(\tau_{D}>t\right)dt\nonumber \\
&\leq e^{d/4}\frac{\sqrt{2}}{\left(8d\right)^{d/4}}\sqrt{\frac{\Gamma\left(d\right)}{\Gamma\left(d/2\right)}}\left(1+\frac{1}{\sqrt{\epsilon}}\right)^{d/2}\int_{a/\lambda_{1}}^{\infty}t^{p-1}e^{-\left(1-\epsilon\right)
\frac{\lambda_1t}{2}dt}\\
%&\leq2^{p}\frac{e^{d/4}\sqrt{2}}{\left(8d\right)^{d/4}}\sqrt{\frac{\Gamma\left(d\right)}{\Gamma\left(d/2\right)}}\left(1+\frac{1}{\sqrt{\epsilon}}\right)^{d/2}\int_{a/\left(2\lambda_{1}\right)}^{\infty}s^{p-1}e^{-\left(1-\epsilon\right)\lambda_{1}s}ds\nonumber \\
% & =2^{p}\frac{e^{d/4}\sqrt{2}}{\left(8d\right)^{d/4}}\sqrt{\frac{\Gamma\left(d\right)}{\Gamma\left(d/2\right)}}\left(1+\frac{1}{\sqrt{\epsilon}}\right)^{d/2}\frac{1}{\left(1-\epsilon\right)^{p}\lambda_{1}^{p}}\int_{\left(1-\epsilon\right)a/2}^{\infty}u^{p-1}e^{-u}du\nonumber \\
 & =2^{p}\frac{e^{d/4}\sqrt{2}}{\left(8d\right)^{d/4}}\sqrt{\frac{\Gamma\left(d\right)}{\Gamma\left(d/2\right)}}\left(1+\frac{1}{\sqrt{\epsilon}}\right)^{d/2}\frac{\Gamma\left(p,\left(1-\epsilon\right)a/2\right)}{\left(1-\epsilon\right)^{p}\lambda_{1}^{p}}.\nonumber
\end{align}
\end{proof}

%\textcolor{blue}{Phanuel: Let me know if this is better? If it still doesn't look like right, then just remove it.}
\begin{remark}
For the interested reader, we remark the following numerical estimate as a  consequence of Theorem \ref{thm-main-thm}.  
In particular, when $p=1$ one can obtain the following  bound
\[
M_{1, 2}(\X)\leq 2C_{1}(2,1)\leq 2f_{1,2}\left(1.65659,0.173247\right)\leq 2\cdot(2.03785)= 4.0757.
\]
where 
\begin{align*}
f_{p,d}\left(a,\epsilon\right):= & \frac{a^{p}}{2^{p}\Gamma\left(p+1\right)} +\frac{1}{\Gamma\left(p\right)}\frac{e^{d/4}\sqrt{2}}{\left(8d\right)^{d/4}}\sqrt{\frac{\Gamma\left(d\right)}{\Gamma\left(d/2\right)}}\left(1+\frac{1}{\sqrt{\epsilon}}\right)^{d/2}\frac{\Gamma\left(p,\left(1-\epsilon\right)a/2\right)}{\left(1-\epsilon\right)^{p}}.
\end{align*}
%Compare this to the bound $2.1063$ obtained in Vogt.
\end{remark}

%The proof of the above lemma will be given in section \ref{sec-lemma-3.2}.

Moreover, we have the following corollary.

\begin{corollary}
\label{cor:Vogt-u1-Sharp} 
We have
\begin{equation}
M_{1, d}(\X)\leq\frac{d}{2}\frac{1}{y_{d}\left(1+\sqrt{y_{d}}\right)}=:2C_3(d,1).\label{eq:VOGTImprovement2} 
\end{equation}
%with the same bound multiplied by $2$ for $M_{1, d}(\E, \lambda)$.  
Here,  $y=y_{d}\in(0,1)$ is the unique solution to 
\begin{equation}\label{eq:g3}
-d+d\sqrt{y}+\left(4+4A_{d}\right)y+(2d)\, y\log\left({\left(1+1/\sqrt{y}\right)/2}\right)=0,\quad y\in\left(0,1\right), 
\end{equation}
where $$A_{d}=\log\left[\frac{2^{d/2}e^{d/4}\sqrt{2}}{\left(8d\right)^{d/4}}\sqrt{\frac{\Gamma\left(d\right)}{\Gamma\left(d/2\right)}}\right],$$
and  $$\lim_{d\to\infty}y_{d}=1.$$
%{\color{blue} check constant}  \textcolor{red}{Yes, the $2^{d/2}$ is needed} 
\end{corollary}

\begin{proof}
From \eqref{eq:C2} we have
\[
C_{1}(d,1):=\inf_{x>0,0<y<1}f(x,y)
\]
where $f:\mathbb{R}_{+}\times\left[0,1\right]\to\mathbb{R}_{+}$ is
defined by
\[
f(x,y)=\frac{x}{2}+\frac{e^{d/4}\sqrt{2}}{\left(8d\right)^{d/4}}\sqrt{\frac{\Gamma\left(d\right)}{\Gamma\left(d/2\right)}}\left(1+\frac{1}{\sqrt{y}}\right)^{d/2}\frac{1}{\left(1-y\right)}e^{-\left(1-y\right)x/2}.
\]
Note that
\[
f_{x}\left(x,y\right)=\frac{1}{2}-\frac{1}{2}\frac{e^{d/4}\sqrt{2}}{\left(8d\right)^{d/4}}\sqrt{\frac{\Gamma\left(d\right)}{\Gamma\left(d/2\right)}}\left(1+\frac{1}{\sqrt{y}}\right)^{d/2}e^{-\left(1-y\right)x/2},
\]
we then obtain the minimizer of $f(\cdot, y)$ at
\[
x_{y}:=\frac{2}{\left(1-y\right)}\log\left[\frac{e^{d/4}\sqrt{2}}{\left(8d\right)^{d/4}}\sqrt{\frac{\Gamma\left(d\right)}{\Gamma\left(d/2\right)}}\left(1+\frac{1}{\sqrt{y}}\right)^{d/2}\right].
\]
We are then led to minimize the
one variable function  
\begin{align*}
 g(y)  &:=f\left(x_{y},y\right)
% & =\frac{1}{2}\frac{2}{1-y}\log\left[\frac{e^{d/4}\sqrt{2}}{\left(8d\right)^{d/4}}\sqrt{\frac{\Gamma\left(d\right)}{\Gamma\left(d/2\right)}}\left(1+\frac{1}{\sqrt{y}}\right)^{d/2}\right]\\
%& +\frac{e^{d/4}\sqrt{2}}{\left(8d\right)^{d/4}}\sqrt{\frac{\Gamma\left(d\right)}{\Gamma\left(d/2\right)}}\left(1+\frac{1}{\sqrt{y}}\right)^{d/2}\frac{1}{\left(1-y\right)}\left(\frac{e^{d/4}\sqrt{2}}{\left(8d\right)^{d/4}}\sqrt{\frac{\Gamma\left(d\right)}{\Gamma\left(d/2\right)}}\left(1+\frac{1}{\sqrt{y}}\right)^{d/2}\right)^{-1}\\
  =\frac{1}{1-y}\log\left[\frac{e^{d/4}\sqrt{2}}{\left(8d\right)^{d/4}}\sqrt{\frac{\Gamma\left(d\right)}{\Gamma\left(d/2\right)}}\left(1+\frac{1}{\sqrt{y}}\right)^{d/2}\right]+\frac{1}{\left(1-y\right)}\\
 &= \frac{\frac{d}{2}\log\left(\frac{1+1/\sqrt{y}}2\right)+1+A_{d}}{1-y}
\end{align*}
where
\[
A_{d}=\log\left[\frac{{2^{d/2}}e^{d/4}\sqrt{2}}{\left(8d\right)^{d/4}}\sqrt{\frac{\Gamma\left(d\right)}{\Gamma\left(d/2\right)}}\right].
\]
Since 
\begin{align}
g'(y) 
% & =\frac{-\frac{d}{4}\left(1-y\right)\frac{y^{-3/2}}{1+1/\sqrt{y}}+\left(\frac{d}{2}\log\left({1+1/\sqrt{y}}\right)+1+A_{d}\right)}{\left(1-y\right)^{2}}\nonumber \\
% & =\frac{-\frac{d}{4}\left(1-\sqrt{y}\right)+\left(\frac{d}{2}\log\left(\frac{1+1/\sqrt{y}}{2}\right)+1+A_{d}\right)y}{\left(1-y\right)^{2}y}\nonumber \\
 & =\frac{-d \left(1-\sqrt{y}\right)+\left((2d)\log\left(\frac{1+1/\sqrt{y}}2\right)+4+4A_{d}\right)y}{4\left(1-y\right)^{2}y},\label{eq:g1}
\end{align}
if we assume that $y_{d}$ is a solution to $g'(y)=0$, 
%\[
%-d+d\sqrt{y}+\left(4+4A_{d}\right)y+(2d)y\log\left({1+1/\sqrt{y}}\right)=0.
%\]
%then 
%\[
%\left(1+A_{d}\right)y=\frac{d}{4}-\frac{d}{4}\sqrt{y}-\frac{d}{2}y\log\left({1+1/\sqrt{y}}\right)
%\]
%and
then
\begin{equation}
\left(1+A_{d}\right)=\frac{d}{4y_d}\left(1-\sqrt{y_d}\right)-\frac{d}{2}\log\left(\frac{1+1/\sqrt{y_d}}2\right).\label{eq:g2}
\end{equation}
Plugging \eqref{eq:g2} back in \eqref{eq:g1} we have that
\begin{align*}
g\left(y_{d}\right) =\frac{d}{4}\frac{1}{y_{d}\left(1+\sqrt{y_{d}}\right)},
\end{align*}
hence we obtain \eqref{eq:VOGTImprovement2}.
Next we show that \eqref{eq:g3} has a  unique solution. Let $F_{d}:\left(0,1\right)\to\mathbb{R}$ be 
\[
F_{d}(y)=-\frac{d}{4}+\frac{d\sqrt{y}}4+y\left(1+A_{d}\right)+\frac{d}{2}\, y\log\left(\frac{1+1/\sqrt{y}}2\right).
\]
We easily find that $\lim_{y\to0}F_{d}(y)=-\frac{d}{4}<0$, $F_{d}(1)=1+A_{d}>0$ and $F_{d}^{\prime}(y)>0$. Therefore the conclusion follows.
\end{proof}

\begin{remark}
From the above corollary we can deduce that $\lim_{d\to\infty}y_{d}=1$. First it can be easily shown that $y_d$ exists (for instance see (3.3) in \cite{Vogt-2019a}). 
%We denote the limit by $y_\infty$, then 
From \eqref{eq:g2} we have
\[
\frac{\left(4+4A_{d}\right)y_{d}}{d}=1-\sqrt{y_{d}}-2y_{d}\log\left(\frac{1+1/\sqrt{y_{d}}}{2}\right).
\]
Taking $d\to\infty$ on both sides we then obtain $y_\infty=1$. This limit coincides with the conclusion in \cite{Vogt-2019a}, but the corollary is sharper comparing to \cite{Vogt-2019a} by providing an almost explicit expression for $y_d$.
\end{remark}

%We shall need the following lemma. We refer the readers to \cite[Lemma 2.6]{Vogt-2019a} for the proof.
%\begin{lemma}
%Let $\Delta_\R$ be the Laplacian on $\R^d$. For $w\in\R^d$ and $\alpha>0$ we have
%\[
%\|e^{-\alpha\rho_w} e^{t\Delta_R}e^{\alpha\rho_w}\|_{2\to2}\le e^{\alpha^2t},\quad 
%\|e^{-\alpha\rho_w} e^{t\Delta_R}e^{\alpha\rho_w}\|_{2\to\infty}\le \left(1+\frac1\beta\right)^{d/4}(8\pi t)^{-d/4}e^{(1+\beta)\alpha^2t}
%\] 
%for all $t, \beta>0$.
%\end{lemma}

%%%%%%%%%%%%%%%%%%%%%%%%%%%%%%%%%%%%%%%%%%%%%%%%%%%%%%%%%%%%%%%%%%%%%%%%%%%%%%%

\section{Sharp asymptotics for $M_{p, d}(\X)$: proof \eqref{asymtotic-bound} of Theorem \ref{asymtotic}} \label{Sec:4}
%and $\mathcal{B}\left(d,p\right)$}

This section concerns the asymptotic estimates for $M_{p, d}(\X)$  in high dimensions.  First, we give an upper bound estimate of $\Mpd$ by analyzing  the variational problem in Theorem \ref{thm-main-thm}, which provides the correct leading order in $d$ for all $p\ge1$.

%The goal of the following result, is to generalize the upper bound
%in (\ref{eq:VOGTImprovement2}) for general $p>1$. This estimate
%is not as sharp as the one given in the previous section, but it will
%allow us to find the leading terms for $M_{p, d}(\E, \lambda)$ and $M_{p, d}(T, \lambda)$,
%for large $d$.
\begin{thm}
 %\textcolor{green}{I don't see where $p\geq 1$ is needed in the proof. Can you confirm theorem can be restate with $p\geq 0$?}
For $p>0$, 
\label{thm:SharpUpper}
\[
M_{p, d}(\X)\leq2^{p}\left(\frac{d}{8}+c\sqrt{d}+1-\frac{1}{1-y_{d}}\right)^{p}C_{2}\left(d,p\right)
\]
%and 
%\[
%M_{p, d}(T, \lambda)\leq\frac{1}{\Gamma\left(p+1\right)}\left(\frac{d}{8}+c\sqrt{d}+1-\frac{1}{1-y_{d}}\right)^{p}C_{2}\left(d,p\right),
%\]
where 
\[
C_{2}(d,p):=1+p\int_{1}^{\infty}u^{p-1}e^{\left(1-u\right)\left[\left(1-y_{d}\right)\left(\frac{d}{8}+c\sqrt{d}+1\right)-1\right]}du,
\]
and
\begin{equation}\label{claim}
 \,\,\,\,c=\frac{1}{4}\sqrt{5\left(1+\frac{1}{4}\log2\right)}, \,\,\,\,\, \text{and}\,\,\,\, y_{d}=\frac{1}{\left(1+\frac{16c}{5\sqrt{d}}\right)^{2}}. 
\end{equation} 
\end{thm}

The proof of Theorem \ref{thm:SharpUpper} will require the following elementary estimate. 
\begin{lemma} With $c$ and $y_d$ as in \eqref{claim} we have 
\begin{equation}\label{Log-Cd-Ineq}
\log\left[2^{\frac{1}{4}-\frac{d}{2}}\left(1+\frac{1}{\sqrt{y_{d}}}\right)^{d/2}\right]+1\leq\left(1-y_{d}\right)\left(\frac{d}{8}+c\sqrt{d}+1\right).
\end{equation}
\end{lemma} 

\begin{proof} First note
\[
\text{LHS}=\frac{1}{4}\log2+\frac{d}{2}\log\left(\frac{1+1/\sqrt{y_{d}}}{2}\right)+1.
\]
Denote by $\gamma:=\frac{8}{5}c$ and set $x=\frac{\gamma}{\sqrt{d}}$. We can easily check that $0<x<1$. Clearly
%\begin{align*}
%y_{d} & =\frac{1}{\left(1+\frac{16c}{5\sqrt{d}}\right)^{2}}=\frac{1}{\left(1+2\frac{8c}{5\sqrt{d}}\right)^{2}}=\frac{1}{\left(1+2x\right)^{2}},
%\end{align*}
%hence 
$1+2x=\frac{1}{\sqrt{y_{d}}}$, 
and hence
\begin{align*}
\text{LHS} %& =\frac{1}{4}\log2+\frac{d}{2}\log\left(\frac{1+1/\sqrt{y_{d}}}{2}\right)+1\\
 & =\frac{5}{4}\gamma^{2}+\frac{d}{2}\log\left(\frac{1+\left(1+2x\right)}{2}\right) =\frac{5}{4}x^{2}d+\frac{d}{2}\log\left(1+x\right).
\end{align*}
On the other hand
\begin{align*}
\text{RHS} %& =\left(1-y_{d}\right)\left(\frac{d}{8}+c\sqrt{d}+1\right)\\
 & =\left(1-\frac{1}{\left(1+2x\right)^{2}}\right)\frac{d}{8}\left(1+\frac{8c}{\sqrt{d}}+\frac{8}{d}\right) =\frac{x+x^{2}}{\left(1+2x\right)^{2}}\frac{d}{2}\left(1+5x+\frac{8}{\gamma^{2}}x^{2}\right).
\end{align*}
Thus it suffices to show that for all $x\in(0,1)$,
\[
\frac{5}{2}x^{2}+\log\left(1+x\right)\leq\frac{x+x^{2}}{\left(1+2x\right)^{2}}\left(1+5x+\frac{8}{\gamma^{2}}x^{2}\right).
\]
This can be shown by elementary calculus. See details in \cite[page 46]{Vogt-2019a}.
\end{proof}

\begin{proof}[Proof of Theorem \ref{thm:SharpUpper}]
Let
\[
f\left(x,y\right):=x^{p}+2^{p}\, p\, C_{d}\left(1+\frac{1}{\sqrt{y}}\right)^{d/2}\frac{1}{\left(1-y\right)^{p}}\Gamma\left(p,\left(1-y\right)x/2\right),
\]
where $C_{d}=\frac{e^{d/4}\sqrt{2}}{\left(8d\right)^{d/4}}\sqrt{\frac{\Gamma\left(d\right)}{\Gamma\left(d/2\right)}}$.
Then 
\[
2^p\Gamma(p+1)C_1(d,p)=\inf_{x>0,0<y<1}f(x,y).
\]
First by letting
\[
f_{x}(x,y)=px^{p-1}\left(1-C_{d}e^{-(1-y)x/2}\left(1+\frac{1}{\sqrt{y}}\right)^{d/2}\right)=0
\]
we obtain the critical point
\begin{equation}
x_{y}=\frac{2}{(1-y)}\log\left[C_{d}\left(1+\frac{1}{\sqrt{y}}\right)^{d/2}\right].\label{eq:Semi2}
\end{equation}
Hence
\begin{align*}
f\left(x_{y},y\right) & =\frac{2^{p}}{\left(1-y\right)^{p}}\left(\log\left[C_{d}\left(1+\frac{1}{\sqrt{y}}\right)^{d/2}\right]\right)^{p}\\
 & +p2^{p}C_{d}\left(1+\frac{1}{\sqrt{y}}\right)^{d/2}\frac{1}{\left(1-y\right)^{p}}\Gamma\left(p,\log\left[C_{d}\left(1+\frac{1}{\sqrt{y}}\right)^{d/2}\right]\right).
\end{align*}
It is known that (for instance see \cite[6.1.18]{Abramowitz-Stegun-1964})
\[
\frac{\Gamma\left(d\right)}{\Gamma\left(d/2\right)}=\frac{\Gamma\left(2(d/2)\right)}{\Gamma\left(d/2\right)}\leq2^{d-1/2}\left(\frac{d}{2e}\right)^{d/2},
\] 
hence we have 
\begin{equation}\label{Cd-Ineq}
C_d\le 2^{-d/2+1/4}.
\end{equation}
%We also have the following claim whose proof is left to the end of this section.
%\begin{claim}\label{Claim-Ineq}
%For  fixed constants
%\begin{equation}\label{Yd-def}
%y_{d}=\frac{1}{\left(1+\frac{16c}{5\sqrt{d}}\right)^{2}}\text{ and }c=\frac{1}{4}\sqrt{5\left(1+\frac{1}{4}\log2\right)},
%\end{equation}
%we have the following inequality,
%\begin{equation}\label{Log-Cd-Ineq}
%\log\left[2^{\frac{1}{4}-\frac{d}{2}}\left(1+\frac{1}{\sqrt{y_{d}}}\right)^{d/2}\right]+1\leq\left(1-y_{d}\right)\left(\frac{d}{8}+c\sqrt{d}+1\right).
%\end{equation}
%\end{claim}
Combining \eqref{Log-Cd-Ineq} and \eqref{Cd-Ineq} we get 
\begin{eqnarray}
\log\left[C_{d}\left(1+\frac{1}{\sqrt{y_{d}}}\right)^{d/2}\right]
&\leq&\log\left[2^{\frac{1}{4}-\frac{d}{2}}\left(1+\frac{1}{\sqrt{y_{d}}}\right)^{d/2}\right]\nonumber\\
&\leq&\left(1-y_{d}\right)\left(\frac{d}{8}+c\sqrt{d}+1\right)-1.\label{eq:Semi3}
\end{eqnarray}
%where $y_d$ and $c$ are given in \eqref{Yd-def}. 
Using $(\ref{eq:Semi3})$ in $f\left(x_{y},y\right)$ we then obtain 
\begin{align}\label{eq:Semi4b}
f\left(x_{y_{d}},y_{d}\right)  \leq2^{p}\left(\left(\frac{d}{8}+c\sqrt{d}+1\right)-\frac{1}{1-y_{d}}\right)^{p} + \frac{p2^{p}}{\left(1-y_{d}\right)^{p}} II, 
% & +p2^{p}C_{d}\left(1+\frac{1}{\sqrt{y_{d}}}\right)^{d/2}\frac{1}{\left(1-y_{d}\right)^{p}}\Gamma\left(p,\log\left[C_{d}\left(1+\frac{1}{\sqrt{y_{d}}}\right)^{d/2}\right]\right)\nonumber \\
% & =2^{p}\left(\frac{d}{8}+c\sqrt{d}+1-\frac{1}{1-y_{d}}\right)^{p}+p2^{p}\frac{1}{\left(1-y_{d}\right)^{p}}II.
\end{align}
where $$II=C_{d}\left(1+\frac{1}{\sqrt{y_{d}}}\right)^{d/2}\Gamma\left(p,\log\left[C_{d}\left(1+\frac{1}{\sqrt{y_{d}}}\right)^{d/2}\right]\right).$$ 
Making the substitution $x=u\log\left[C_{d}\left(1+\frac{1}{\sqrt{y_{d}}}\right)^{d/2}\right]$ and plugging in \eqref{eq:Semi3} we have
\begin{align}
II %& =C_{d}\left(1+\frac{1}{\sqrt{y_{d}}}\right)^{d/2}\int_{\log\left[C_{d}\left(1+\frac{1}{\sqrt{y_{d}}}\right)^{d/2}\right]}^{\infty}x^{p-1}e^{-x}dx\nonumber \\
 & =C_{d}\left(1+\frac{1}{\sqrt{y_{d}}}\right)^{d/2}\left(\log\left[C_{d}\left(1+\frac{1}{\sqrt{y_{d}}}\right)^{d/2}\right]\right)^{p}\\&\times \left[\int_{1}^{\infty}u^{p-1}\left(C_{d}\left(1+\frac{1}{\sqrt{y_{d}}}\right)^{d/2}\right)^{-u}du\right]\nonumber \\
 & \leq %C_{d}\left(1+\frac{1}{\sqrt{y_{d}}}\right)^{d/2}\left(\left(1-y_{d}\right)\left(\frac{d}{8}+c\sqrt{d}+1\right)-1\right)^{p}\int_{1}^{\infty}u^{p-1}\left(C_{d}\left(1+\frac{1}{\sqrt{y_{d}}}\right)^{d/2}\right)^{-u}du\nonumber \\
 %& =
 C_{d}\left(1+\frac{1}{\sqrt{y_{d}}}\right)^{d/2}\left(1-y_{d}\right)^{p}\left(\left(\frac{d}{8}+c\sqrt{d}+1\right)-\frac{1}{\left(1-y_{d}\right)}\right)^{p}\\
&\times \left[\int_{1}^{\infty}u^{p-1}\left(C_{d}\left(1+\frac{1}{\sqrt{y_{d}}}\right)^{d/2}\right)^{-u}du\right].\nonumber \label{eq:IIu1}
\end{align}
Moreover, clearly from \eqref{eq:Semi3} we have 
\begin{equation}
\left(C_{d}\left(1+\frac{1}{\sqrt{y_{d}}}\right)^{d/2}\right)^{-u}\leq e^{-u\left[\left(1-y_{d}\right)\left(\frac{d}{8}+c\sqrt{d}+1\right)-1\right]}.\label{eq:IIu2}
\end{equation}
Hence
\begin{equation}
II\leq\left(1-y_{d}\right)^{p}\left(\left(\frac{d}{8}+c\sqrt{d}+1\right)-\frac{1}{1-y_{d}}\right)^{p}\int_{1}^{\infty}u^{p-1}e^{\left(1-u\right)\left[\left(1-y_{d}\right)\left(\frac{d}{8}+c\sqrt{d}+1\right)-1\right]}du.\label{eq:Semi4}
\end{equation}
Using $(\ref{eq:Semi4})$ in $(\ref{eq:Semi4b})$ we arrive at
\begin{align*}
f\left(x_{y_{d}},y_{d}\right) &\leq
% 2^{p}\left(\frac{d}{8}+c\sqrt{d}+1-\frac{1}{1-y_{d}}\right)^{p}\\
% & +p2^{p}\frac{1}{\left(1-y_{d}\right)^{p}}\left(1-y_{d}\right)^{p}\left(\frac{d}{8}+c\sqrt{d}+1-\frac{1}{1-y_{d}}\right)^{p}\int_{1}^{\infty}u^{p-1}e^{\left(1-u\right)\left[\left(1-y_{d}\right)\left(\frac{d}{8}+c\sqrt{d}+1\right)-1\right]}du\\
 %& =
 2^{p}\left(\frac{d}{8}+c\sqrt{d}+1-\frac{1}{1-y_{d}}\right)^{p}\left[1+p\int_{1}^{\infty}u^{p-1}e^{\left(1-u\right)\left[\left(1-y_{d}\right)\left(\frac{d}{8}+c\sqrt{d}+1\right)-1\right]}du\right].
\end{align*}
\end{proof}
In the lemma below we show that our result is indeed sharp, by comparing to a unit ball. 

\begin{lemma}
\label{Lemma-Ball}
%\textcolor{green}{ok here $p\geq 1$ is needed due to Jensen's} 
Let $B\left(0,1\right)\subset\mathbb{R}^{d}$ be
the unit ball centered at zero, then 
\[
\left(\frac{d}{4}\right)^{p}\leq\lambda_{1}^{p}\left(B\left(0,1\right)\right)\sup_{x\in B\left(0,1\right)}\mathbb{E}_{x}\left[\tau_{B\left(0,1\right)}^{p}\right],
\]
for $p\geq 1$. 
\end{lemma}

\begin{proof}
It is well known that $\lambda_{1}\left(B(0,1)\right)\geq\frac{d^{2}}{4}$ (for instance, see  \cite{Lorch}).
By a simple calculation we have that $\mathbb{E}_{x}\left[\tau_{B(0,1)}\right]=\frac{1-\left|x\right|^{2}}{d}.$
Hence $$\sup_{x\in B\left(0,1\right)}\mathbb{E}_{x}\left[\tau_{B\left(0,1\right)}\right]=\mathbb{E}_{0}\left[\tau_{B\left(0,1\right)}\right]=\frac{1}{d}.$$
By Jensen's inequality we have 
\begin{align*}
\lambda_{1}^{p}\left(B\left(0,1\right)\right)\cdot\mathbb{E}_{0}\left[\tau_{B\left(0,1\right)}^{p}\right] & \geq\lambda_{1}^{p}\left(B\left(0,1\right)\right)\cdot\left(\mathbb{E}_{0}\left[\tau_{B\left(0,1\right)}\right]\right)^{p}=\frac{d^{p}}{4^{p}}.
\end{align*}
\end{proof}

\begin{proof}[Proof of \eqref{asymtotic-bound} of Theorem \ref{asymtotic}]
From Theorem \ref{thm:SharpUpper} we have
\[
M_{p,d}(\X)\leq2^{p}\left(\frac{d}{8}+c\sqrt{d}+1-\frac{1}{1-y_{d}}\right)^{p}C_{2}\left(d,p\right),
\]
where 
\[
C_{2}(d,p):=1+p\int_{1}^{\infty}u^{p-1}e^{\left(1-u\right)\left[\left(1-y_{d}\right)\left(\frac{d}{8}+c\sqrt{d}+1\right)-1\right]}du,
\]
and 
\[
y_{d}=\frac{1}{\left(1+\frac{16c}{5\sqrt{d}}\right)^{2}},\,\,\,c=\frac{1}{4}\sqrt{5\left(1+\frac{1}{4}\log2\right)}.
\]
First we claim that $\lim_{d\to\infty}C_2(d,p)=1$. Note that 
\[
\left(1-y_{d}\right)\left(\frac{d}{8}+c\sqrt{d}+1\right)-1\ge 1+\frac{4c}{5}\sqrt{d}\to\infty
%=\left(1-\frac{1}{\left(1+\frac{16c}{5\sqrt{d}}\right)^{2}}\right)\left(\frac{d}{8}+c\sqrt{d}+1\right)-1
\]
as $d\to\infty$. 
Hence when $u\ge 1$ we have
\[
\lim_{d\to\infty}u^{p-1}e^{\left(1-u\right)\left[\left(1-y_{d}\right)\left(\frac{d}{8}+c\sqrt{d}+1\right)-1\right]}=0,
\]
Moreover, since
\[
\int_{1}^{\infty}u^{p-1}e^{\left(1-u\right)\left[\left(1-y_{d}\right)\left(\frac{d}{8}+c\sqrt{d}+1\right)-1\right]}du
\le
\int_{1}^{\infty}u^{p-1}e^{\left(1-u\right)}du\le  e\Gamma(p),
\]
by the dominated convergence theorem we obtain that
\begin{equation}
\lim_{d\to\infty}\int_{1}^{\infty}u^{p-1}e^{\left(1-u\right)\left[\left(1-y_{d}\right)\left(\frac{d}{8}+c\sqrt{d}+1\right)-1\right]}du=0.\label{eq:MCT}
\end{equation}
%so that $\lim_{d\to\infty}C_{2}\left(d,p\right)=1$. 
%
%To see the monotone convergence theorem argument in (\ref{eq:MCT})
%we define
%
%\[
%f_{d}(u)=e^{\left(1-u\right)\left[\left(1-y_{d}\right)\left(\frac{d}{8}+c\sqrt{d}+1\right)-1\right]},
%\]
%and want to show that $f_{d+1}(u)\leq f_{d}(u)$ for all $u\geq1$
%for large enough $d$. Since $f_{d}(u)=e^{\left(1-u\right)g(d)}$
%where $g(d)=\left[\left(1-y_{d}\right)\left(\frac{d}{8}+c\sqrt{d}+1\right)-1\right]$
%, then it suffices to show that $g(d)$ is increasing for large enough
%$d$. Note that 
%\begin{align*}
%g^{\prime}(d) & =\frac{2c\left(1024c^3+1216c^{2}\sqrt{d}+240cd+25\left(d-8\right)\sqrt{d}\right)}{\sqrt{d}\left(16c+5\sqrt{d}\right)^{3}}\geq0,
%\end{align*}
%for $d\geq8$, which shows $g(d+1)\geq g(d)$ for $d\geq8$. This
%completes the monotone convergence theorem argument since $h_{d}(u)=u^{p-1}f_{d}(u)\downarrow0$
%and $h_{d}(u)\geq0$. 
It now follows readily that 
\[
\limsup_{d\to\infty}\frac{M_{p,d}(\X)}{d^{p}}\leq\frac{1}{4^{p}}.
\]
%Since the unit ball $B\left(0,1\right)\subset\mathbb{R}^{d}$ satisfies
%\[
%\frac{1}{4^{p}}\leq\frac{M_{p,d}(\E, \lambda)}{d^{p}},
%\]
Together with Lemma \ref{Lemma-Ball}  we then obtain that $$
\frac{1}{4^{p}}\leq\liminf_{d\to\infty}\frac{M_{p,d}(\X)}{d^{p}},
$$
and concludes the proof of the asymptotic bound \eqref{asymtotic-bound} of Theorem \ref{asymtotic}.

%as needed. 
%Using the fact that $2^p\Gamma(p+1)M_{p, d}(T, \lambda)=M_{p,d}(\E, \lambda)$, we get the conclusion for the $p$-torsion function $u_p$.  
\end{proof}

%%%%%%%%%%%%%%%%%%%%%%%%%%%%%%%%%%%%%%%%%%%%%%%%%%%%%%%%%%%%%%%%%%%%%%%%%%%%%%%%%%%%%%%%%%%%%%%%%%%%%%%%%%%%%%%%%%%%%%%%%%%%%%%%%%%%%%%%%%%%%%%%%%%%%%%%%%%%%%%%%%%%%%%%%%%%%%

\section{Lower bound for $\mpd$: proof of \eqref{LowerBound-Exittime-p} of Theorem \ref{asymtotic}}\label{sec:LowerBound}

We remark that the lower bound in \eqref{LowerBound-Exittime-p} for $p=1$ has been known for many years, as mentioned in  \cite{Banuelos-Carrol1994}.

\begin{proof}[Proof of \eqref{LowerBound-Exittime-p} of Theorem \ref{asymtotic}]

We first prove the inequality. Let us assume for the moment that the domain $D$ is bounded (or even just that it has finite volume). In this case we have a discrete spectrum with a complete set of eigenfunctions on $L^2(D)$  and the eigenfunction $\varphi_1$  corresponding to $\lambda_1(D)$ is in $L^{\infty}(D)$.  For this, we refer the reader to \cite{Davies1989a}.    Since

\begin{equation}\label{eq:EigenfuncEq1}
e^{-\lambda_{1}t/2}\varphi_{1}(x)=\int_{D}p_{D}(x,y,t)\varphi_{1}(y)dy
\end{equation}
integrating in time we find that
\begin{align}
\varphi_1(x)\frac{2^{p}p}{\lambda_{1}^{p}(D)}\Gamma\left(p\right)&=\varphi_{1}(x)\int_{0}^{\infty}pt^{p-1}e^{-\lambda_{1}t/2}dt\\
 & =p\int_{0}^{\infty}\int_{D}t^{p-1}p_{D}(x,y,t)\varphi_{1}(y)dydt\nonumber \\
 & \leq \sup_{y\in D}\varphi_{1}(y)\cdot \bigg(p\int_{0}^{\infty}\int_{D}t^{p-1}p_{D}(x,y,t)dydt\bigg)\nonumber \\
 & =\sup_{y\in D}\varphi_{1}(y)\, \mathbb{E}_{x}\left[\tau_{D}^{p}\right].   \label{eq:LowerBound-1}
\end{align}
Since $p\Gamma(p)=\Gamma(p+1)$, this gives the desired lower bound by taking a supremum over all $x\in D$. 

%To remove the $L^{\infty}$ assumption on $\varphi_1$, let us recall that if the domain $D$ is bounded (or merely of finite volume), we have a complete set of orthonormal eigenfunctions on $L^2(D)$. In particular, $\varphi_1\in L^2(D)$ (see for example, \cite{Davies1989a}). Then for any $x\in D$
%\begin{align}
%\varphi_{1}(x) & =e^{\lambda_{1}t/2}\int_{D}p_{D}(x,y,t)\varphi_{1}(y)dy \leq e^{\lambda_{1}t/2}\left(\int_{D}p_{D}(x,y,t)^{2}dy\right)^{1/2}\left\Vert \varphi_{1}\right\Vert _{2}\nonumber \\
% & =e^{\lambda_{1}t/2}p_{D}\left(x,x,2t\right)^{1/2}\leq e^{\lambda_{1}t/2}p_{\mathbb{R}^{d}}\left(x,x,2t\right)^{1/2}\nonumber \\
% & \leq e^{\lambda_{1}t/2}/\sqrt{4\pi t}<\infty.\label{eq:1stEigenfunctionMax}
%\end{align}

To remove the boundedness assumption on $D$,   let $r>0$ and consider
the open set $D\cap B\left(0,r\right)$ which is nonempty for large
enough $r$. 
Since $D\cap B(0,r)\subset D$, we have $\mathbb{E}_{x}\left[\tau_{D\cap B(0,r)}^{p}\right]\leq\mathbb{E}_{x}\left[\tau_{D}^{p}\right]$
and it follows that 
\[
\sup_{x\in D\cap B\left(0,r\right)}\mathbb{E}_{x}\left[\tau_{D}^{p}\right]\geq\sup_{x\in D\cap B\left(0,r\right)}\mathbb{E}_{x}\left[\tau_{D\cap B(0,r)}^{p}\right]\geq2^{p}\Gamma\left(p+1\right)\cdot\bigg(\lambda_{1}\left(D\cap B\left(0,r\right)\right)\bigg)^{-p}.
\]
Taking $r\to\infty$ completes the proof of the lower bound. 
%\end{proof}

It remains to prove the sharpness of \eqref{LowerBound-Exittime-p} for integers $p$.  For any $d\geq 2$, it is shown in \cite[Theorem 1]{Vandenberg-2017a} and \cite[Theorem 3.3]{Henrot-Lucardesi-Philippin-2018} that there  exists a sequence of bounded domains $D_{\epsilon_{n}}\subset\mathbb{R}^{d}$
satisfying

\begin{equation}
2\leq \lambda_1(D_{\epsilon_{n}})\sup_{x\in D_{\epsilon_{n}}} \E_x[\tau_{D_{\epsilon_{n}}} ]<2+\epsilon_{n},\label{VandenVerg-Sets}
\end{equation}
where $\epsilon_{n}\to0$, as $n\to\infty$.  
To finish, we need the following  inequality whose proof we provide here for completeness. (See for example \cite[Corollary 1]{Banuelos-1987} and \cite[Lemma 18.1]{Burkholder-1973})

\begin{lemma}
\label{lem:FactorialMomentBound}Let $D\subset\mathbb{R}^{d}$ be
a domain satisfying $\sup_{x\in D}\mathbb{E}_{x}\left[\tau_{D}\right]<\infty$. 
Then for any $k\in\mathbb{N}$, 
\[
\mathbb{E}_{x}\left[\tau_{D}^{k}\right]\leq k!\left(\sup_{x\in D}\mathbb{E}_{x}\left[\tau_{D}\right]\right)^{k},\,\,\,\,\,\,\, x\in D. 
\]
\end{lemma}

\begin{proof}
By the Markov property and Fubini's theorem we have for any $a\geq 0$, 
\begin{align}
&\int_{a}^{\infty}\mathbb{P}_{x}\left(\tau_{D}>t\right)dt  =\int_{0}^{\infty}\mathbb{P}_{x}\left(\tau_{D}>t+a\right)dt\nonumber \\
 &\qquad =\int_{0}^{\infty}\mathbb{E}_{x}\left[1_{\left(\tau_{D}>a\right)}\mathbb{P}_{X_{a}}\left(\tau_{D}>t\right)\right]dt
  =\mathbb{E}_{x}\left[1_{\left(\tau_{D}>a\right)}\mathbb{E}_{X_{a}}\left[\tau_{D}\right]\right]\nonumber \\
 &\qquad \leq\left(\sup_{x\in D}\mathbb{E}_{x}\left[\tau_{D}\right]\right)\mathbb{P}_{x}\left(\tau_{D}>a\right)\label{eq:MomentBound1}.
\end{align}
Multiplying both sides by $ka^{k-1}$ and integrating on $a$ gives that 
\begin{align*}
&\int_{0}^{\infty}ka^{k-1}\int_{a}^{\infty}\mathbb{P}_{x}\left(\tau_{D}>t\right)dtda 
% =\int_{0}^{\infty}\int_{0}^{t}ka^{k-1}\mathbb{P}_{x}\left(\tau_{D}>t\right)dadt\\
=\int_0^{\infty} t^k\P_x\left(\tau_D>t\right) dt
 =\frac{1}{k+1}\mathbb{E}_{x}\left[\tau_{D}^{k+1}\right],
\end{align*}
and
\[
\left(\sup_{x\in D}\mathbb{E}_{x}\left[\tau_{D}\right]\right)\int_{0}^{\infty}ka^{k-1}\mathbb{P}_{x}\left(\tau_{D}>a\right)da=\left(\sup_{x\in D}\mathbb{E}_{x}\left[\tau_{D}\right]\right)\mathbb{E}_{x}\left[\tau_{D}^{k}\right].
\]
The desired inequality then follows by induction. 
\end{proof}

%\begin{prop}
Returning to the sharpness of inequality \eqref{LowerBound-Exittime-p}, fix $k\in\mathbb{N}$.  Let $D_{\epsilon_{n}}$ be the domains satisfying  \eqref{VandenVerg-Sets}. We claim that 
\[
\lambda_{1}\left(D_{\epsilon_{n}}\right)^{k}\cdot\sup_{x\in D_{\epsilon_{n}}}\mathbb{E}_{x}\left[\tau_{D_{\epsilon_{n}}}^{k}\right]\leq 2^kk!+2k\cdot k!\epsilon_n+o(\epsilon_n)
\]
where $\epsilon_{n}\to0$ as $n\to\infty$. 
%Hence the lower bounds $(\ref{LowerBound-Exittime-p})$ 
%is sharp for integer moments.
%\end{prop}

%\begin{proof}
 Indeed, from  Lemma \ref{lem:FactorialMomentBound} and the estimate $(\ref{VandenVerg-Sets})$. 
\begin{align*}
&\lambda_{1}\left(D_{\epsilon_{n}}\right)^{k}\cdot\sup_{x\in D_{\epsilon_{n}}}\mathbb{E}_{x}\left[\tau_{D_{\epsilon_{n}}}^{k}\right] \leq\lambda_{1}\left(D_{\epsilon_{n}}\right)^{k}k!\left(\sup_{x\in D_{\epsilon_{n}}}\mathbb{E}_{x}\left[\tau_{D_{\epsilon_{n}}}\right]\right)^{k}\\
 &\qquad =k!\left(\lambda_{1}\left(D_{\epsilon_n}\right)\cdot\sup_{x\in D_{\epsilon_{n}}}\mathbb{E}_{x}\left[\tau_{D_{\epsilon_{n}}}\right]\right)^{k}
  \leq k!\left(2+\epsilon_{n}\right)^{k}\\
 & \qquad= 2^kk!+2k\cdot k!\epsilon_n+o(\epsilon_n).\\
\end{align*}
%{ \color{blue}Jing: In the last inequality, shouldn't it be $2^kk!+2k\cdot k!\epsilon_n+o(\epsilon_n)$? So the upper bound of the conclusion can be $2^k\Gamma(k+1)+C_k\epsilon_n$ for some constant $C_k>0$. For instance $C_k$ can be $2^k\Gamma(k+2)$.}
%\end{proof}
This proves the sharpness of the inequality \eqref{LowerBound-Exittime-p} and completes the proof of the Theorem. 
\end{proof}
It is reasonable to conjecture that under the same assumptions as in Lemma \ref{lem:FactorialMomentBound},
the inequality  
\begin{equation}
\mathbb{E}_{x}\left[\tau_{D}^{p}\right]\leq\Gamma\left(p+1\right)\left(\sup_{x\in D}\mathbb{E}_{x}\left[\tau_{D}\right]\right)^{p}\label{eq:GeneralpMomentbound}
\end{equation}
holds for any  $p\geq 1$.   
This leads us to the following conjecture.

\begin{conjecture}\label{ConjecLower}
The lower bounds $(\ref{LowerBound-Exittime-p})$ 
is sharp for any $p\geq 1$.
\end{conjecture}

%%%%%%%%%%%%%%%%%%%%%%%%%%%%%%%%%%%%%%%%%%%%%%%%%%%%%%%%%%%%%%%%%%%%%%%%%%%%%%%%%%%%%%%%%%%%%

\section{Extremal Domains for $M_{p, d}(\cC)$}\label{Sec:Extremal}
%Proof of Theorem \ref{Thm:Existence}}\label{Sec:Extremal}

The main goal of this section is to prove that the shape functional $$G_{p,d}\left(D\right)  =\lambda_{1}^{p}\left(D\right)\sup_{x\in D}\mathbb{E}_{x}\left[\tau_{D}^{p}\right] $$
admits a maximizer in
the class of bounded convex domains
in $\mathbb{R}^{d}$ or in the class of convex domains which are symmetric
with respect to each coordinate axis.

\subsection{Motivation and preliminary discussions}\label{Sec:5.1}
While balls appear to be extremals for several isoperimetric type inequalities (including the classical ones \eqref{egen1} and \eqref{Lifetime}), surprisingly they are not extremals for $\Mpd$. For instance it is observed in \cite[pg. 599]{Banuelos-Carrol1994} that 
\begin{equation}\label{B-vs-T}
\lambda_1\left(B\right)\sup_{x\in B}\E_{x}\left[\tau_{B}\right]< \lambda_1\left(\T\right)\sup_{x\in \T}\E_{x}\left[\tau_{\T}\right],
\end{equation}
where $\T$ is is the equilateral triangle (see also  \cite[Corollary 3.7]{Henrot-Lucardesi-Philippin-2018}). 
%{\color{blue}what is $T$? We used $T$ for the functional of torsion function, do we want to use it here as a domain?}
Moreover it  was conjectured in \cite{Henrot-Lucardesi-Philippin-2018} that no extremal domain exists over
the class of all domains. 
% for the torsion quantity $M_{1, d}(T, \lambda)$ and, of course, this would be the same for $M_{1, d}(\E, \lambda)$. 
%The evidence provided in \cite{Henrot-Lucardesi-Philippin-2018} states that  in
%the class of bounded open $C^{2}$ domains, if $\lambda_{1}\left(D\right)\left\Vert u_{D}\right\Vert _{\infty}$ is differentiable
%at $D$, then $D$ is not a maximizer. 

Therefore it is reasonable, when looking for extremals, to restrict the class of domains. 
When restricted to the class of convex domains, Payne showed in \cite{Payne-1981} that
\begin{equation}\label{PayneConvex}
m_{1, d}(\cC)=\frac{\pi^{2}}{4}.
\end{equation}
From this it follows trivially that the minimizer domain over convex domains is given by the infinite slab $S_d=\mathbb{R}^{d-1}\times\left(-1,1\right)$.  The existence of extremal for $M_{1, d}(\cC)$ is proved in  \cite{Henrot-Lucardesi-Philippin-2018}. The authors further conjectured that when 
$d=2$, the equilateral triangle $\T$ is an extremal. That is,  
%\begin{conjecture}[Convex Domains, \cite{Henrot-Lucardesi-Philippin-2018}]\label{ConvexConjecture}  With the supremum in \eqref{maxEigenExpectred} taken over all convex domains,  
\begin{equation}\label{ConvexConj}
M_{1, 2}(\cC)=  \lambda_1\left(\T\right)\sup_{x\in \T}\E_{x}\left[\tau_{\T}\right].
\end{equation}
%where as above $\T$ is the equilateral triangle. 
%\end{conjecture} 
%We conjecture that the same should hold for  any $p>1$.  {\color{red}Jing: Can we elaborate more here? Otherwise we should delete this sentence.}

\subsection{Motivation, symmetric convex domains}\label{Sec:5.2}

% we mention the {\it hot spots} conjecture in \cite{Banuelos-Burdzy-1999, Banuelos-Pang-Pascu-2004,Pascu-2002,Jerison-Nadirashvili-2000}. Another example is the \emph{Fundamental Gap} Conjecture which was verified under  symmetry assumptions on the domain in \cite{Banuelos-Kulczycki2006,Banuelos-Mendez-Hernandez-2000,Davis-2001} before its full solution in \cite{Andrews-Clutterbuck-2011}.

Another class of domains that is worth investigation is the class of doubly symmetric planar domains $\SC$. %There have been many interesting problems concerning the geometry of the Laplacian remain open for planar convex domains, but have been made substantial progress for doubly symmetric convex domains  \cite{Banuelos-Burdzy-1999, Banuelos-Pang-Pascu-2004,Pascu-2002,Jerison-Nadirashvili-2000, Banuelos-Kulczycki2006,Banuelos-Mendez-Hernandez-2000,Davis-2001, Andrews-Clutterbuck-2011}.

Regarding the question of extremals, it is not hard to see the ball fails again in the class $\SC$. In fact, with brief computations below we can show that 
\begin{equation}\label{Ball-vs-Square}
\lambda_{1}(B)\sup_{x\in B}\E_{x}\left[\tau_{B}\right]< \lambda_{1}(Q_2)\sup_{x\in Q_2}\E_{x}\left[\tau_{Q_2}\right].
\end{equation}
%$
%G(B)\leq G(S).
%$
First note that in both cases 
$$\sup_{x\in B}\E_{x}\left[\tau_{B}\right]=\E_{(0, 0)}\left[\tau_{B}\right]$$ and $$\sup_{x\in Q_2}\E_{x}\left[\tau_{Q_2}\right]=\E_{(0,0)} \left[\tau_{Q_2}\right].$$  Furthermore,   
\begin{equation}\label{BallCom} 
\lambda_{1}(B){\E_{(0,0)}}\left[\tau_{B}\right]=\frac{j_{0}^{2}}{2}\approx2.8916,
\end{equation} 
where $j_0$ is the first positive root of the first Bessel function. 

%Also, shouldn't we use $j_0$, to be consistent with later computations? --Jing}
On the other hand, $\lambda_{1}\left(Q_2\right)=\frac{\pi^{2}}{2}$ and by independence,  
$$
\P_{(0, 0)}(\tau_{Q_2}>t)=\P_0(\tau_I>t)\P_0(\tau_I>t),
$$
where $I=(-1,1)$.  The eigenfunction expansion for the heat kernel for the interval $I$ (see \cite{Coffey-2012,Markowsky-2011}) leads to the formula  
\begin{equation}\label{Exit-Interval} 
\mathbb{E}_{(0, 0)}\left[\tau_{Q_2}\right]=\left[1-\frac{32}{\pi^{3}}\sum_{n=0}^{\infty}\frac{\left(-1\right)^{n}}{\left(2n+1\right)^{3}}\text{sech}\left[\left(n+\frac{1}{2}\right)\pi\right]\right]. 
\end{equation}
Thus 
\begin{align}\label{SquareCom} 
\lambda_{1}(Q_2)\mathbb{E}_{0}\left[\tau_{Q_2}\right]
 & =\frac{\pi^{2}}{2}\left[1-\frac{32}{\pi^{3}}\sum_{n=0}^{\infty}\frac{\left(-1\right)^{n}}{\left(2n+1\right)^{3}}\text{sech}\left[\left(n+\frac{1}{2}\right)\pi\right]\right]\\
 & \approx2.90843\nonumber
\end{align}
%Hence $G(B)\leq G(S)$. 
which verifies \eqref{Ball-vs-Square}. 

Given conjecture \eqref{ConvexConj}, it is reasonable to conjecture that the extremal for $M_{1, 2}(\SC)$, if exists, is given by the square $Q_2:=\{(x, y),  |x|<1, |y|<1\}$.
In the next two sections we will focus on proving the existence of the extremals for both classes $\cC$ and $\SC$.
%\[
%M_{1, 2}(\SC)= \lambda_1^p(Q_2)\E_{(0,0)}[\tau_{Q_2}].
%\]

%For the problem discussed in this paper we have the following conjecture. Let 
%$$
%Q_{d}=\{(x_1, x_2, \dots x_d)\in \R^d: |x_{i}|<1\}
%$$ 
%denote the unit cube in $\R^d$.  
%When $d=2$ we will simply write $Q_2=S$. 
% note that $G$ is invariant under
%scaling. 
 %\begin{conjecture}[Doubly Symmetric Convex] 
%\label{Conjecture-Double-Symmetric} 
%with the supremum in \eqref{maxEigenExpectred} taken over all {\it doubly symmetric} planar convex domains, we have 

\subsection{Preliminary results}\label{Sec:5.3}
%As it was done in \cite[Theorem 3.2]{Henrot-Lucardesi-Philippin-2018}, we show that an extremal domain does exist in the class of doubly symmetric convex domains. More generally, we consider the class of convex domains in $\mathbb{R}^d$ that are symmetric with respect to each coordinate axes. 

For any given domain $D\subset \R^d$, define  
$
\M_{p}\left(D\right)  =\sup_{x\in D}\mathbb{E}_{x}\left[\tau_{D}^{p}\right]
$ so that our function from \eqref{eq-shape-func} becomes 
\begin{align*} 
G_{p, d}\left(D\right) & =\lambda_{1}^{p}\left(D\right)\sup_{x\in D}\mathbb{E}_{x}\left[\tau_{D}^{p}\right]=\lambda_{1}^{p}\left(D\right)\M_{p}\left(D\right).
\end{align*}

%\begin{df}
%Let $\mathcal{C}$ be the class of bounded convex domains
%in $\mathbb{R}^{d}$. Let $\mathcal{SC}$ be the subclass of domains in $\mathcal{C}$ 
%%\textbf{open
%%bounded convex sets} 
%that are \textbf{symmetric
%with respect to each coordinate axes}.
%\end{df}
%Let $d_{\mathcal{H}}$ be the \textbf{Hausdorff}

Recall that for any two compact sets $K_1,K_2\subset\mathbb{R}^{d}$ we  define the \emph{Hausdorff
distance} $d_{\mathcal{H}}$ by
\[
d_{\mathcal{H}}\left(K_1,K_2\right)=\max\left\{ \sup_{x\in  K_1}d\left(x,K_2\right),\sup_{y\in K_2}d\left(K_1,y\right)\right\},
\]
where $d(\cdot,\cdot)$ denotes the Euclidean distance in $\R^d$.
Therefore for any bounded open sets $A,B\subset\R^d$ we have that 
\begin{equation}\label{eq-Haus}
d_{\mathcal{H}}\left(A,B\right)=\max\left\{ \sup_{x\in B\backslash A}\inf_{y\in\partial B}d\left(x,y\right),\sup_{x\in A\backslash B}\inf_{y\in\partial A}d\left(x,y\right)\right\} .
\end{equation}
%{\color{red}Jing: I'm a little confused by this definition. It's not
%\begin{equation}\label{eq-Haus}
%d_{\mathcal{H}}\left(A,B\right)=\max\left\{ \sup_{x\in  A}d\left(x,B\right),\sup_{y\in B}d\left(A,y\right)\right\} ?
%\end{equation}
%The two distances don't seem to be the same to me}
 This definition is given by \cite[Corollary 2.2.13]{Henrot-Pierre-2018}.
In the sequel we use the fact that inclusion is stable under
convergence with respect to $d_{\mathcal{H}}$. That is, take sets $U_n\subset D_n \subset \mathbb{R}^d$ for all $n$. If $U_{n}\to U$
with respect to $d_{\mathcal{H}}$ and $D_{n}\to D$ with respect
to $d_{\mathcal{H}}$, then $U\subset D$.

%We prove the following preliminary lemmas for the proof of the existence of extremals for the functional $M_p(D)$ in a suitable class of domains. 

\begin{lemma}
\label{Lem:Existence1}If a sequence $\{D_{n}\}_{n=1}^\infty$ in $\mathcal{SC}$ converges to a set $D\in\mathcal{C}$
with respect to the Hausdorff metric, then $D\in \mathcal{SC}$. 
\end{lemma}

\begin{proof}
Note that $D$ is open. Take any $x=\left(x_{1},x_{2},\dots,x_{d}\right)\in D$, then $x\in D_{n}$
for $n$ large enough. Since $D_{n}$ is symmetric then $\left(-x_{1},\dots,x_{d}\right),\left(x_{1},-x_{2},\dots,x_{d}\right),\dots,\left(x_{1},x_{2},\dots,-x_{d}\right)\in D_{n}$
for $n$ large enough. Since inclusion is stable under limits of the
Hausdorff distance then $$\left(-x_{1},\dots,x_{d}\right),\left(x_{1},-x_{2},\dots,x_{d}\right),\dots,\left(x_{1},x_{2},\dots,-x_{d}\right)\in D$$
as well. This shows $D$ is symmetric with respect to all axes. Convexity
is well know. 
\end{proof}
We will need the following key estimates on the $p-$moments of exit times in order to prove that $M_p$ is continuous in the class $\mathcal{SC}$ and $\mathcal{C}$.

\begin{lemma}\label{Lem:Existence1a}
Suppose $U\subset D\subset \mathbb{R}^{d}$, where $D$ is a bounded Lipschitz domain and $U$ is a domain.  
\begin{itemize} 
\item[(i)] If $p\geq 1$, then 
\begin{equation}
\sup_{x\in D}\mathbb{E}_{x}\left[\left(\tau_{D}-\tau_{U}\right)^{p}\right]\leq C_{p,D}\sup_{x\in D\backslash{U}}\left(d\left(x,\partial D\right)\right)^{\beta}.\label{lip1}
\end{equation} 
\item [(ii)] If $0<p<1$, then 
\begin{equation}
\sup_{x\in D}\mathbb{E}_{x}\left[\left(\tau_{D}-\tau_{U}\right)^{p}\right]\leq C_{\beta,D}\sup_{x\in D\backslash{U}}\left(d\left(x,\partial D\right)\right)^{\beta p}.\label{lip2}
\end{equation} 
\end{itemize}
Here $\beta>0$ depends on the Lipschitz character of the domain.

\end{lemma}

\begin{proof}
Take $x\in U$. By the strong Markov property we have for any $p>0$,
\begin{align}
\mathbb{E}_{x}\left[\left(\tau_{D}-\tau_{U}\right)^{p}\right] & =\mathbb{E}_{x}\left[\mathbb{E}_{B_{\tau_{U}}}\left[\tau_{D}^{p}\right]\right]\nonumber \\
 & \leq\sup_{x\in\partial U}\mathbb{E}_{x}\left[\tau_{D}^{p}\right].\label{Eig-Haus-1}
\end{align}
%Since $D$ is $C^{1}$ then for any $x\in D$ ,
%\begin{equation}
%E_{x}\left[\tau_{D}\right]\leq C_{D}d\left(x,\partial D\right)\label{Eig-Haus-2}. 
%\end{equation}
%and for the ground state first eigenfunction $\phi$ we also have
%that 
%\begin{equation}
%\phi(x)\leq C_{D}d\left(x,\partial D\right).\label{Eig-Haus-3}
%\end{equation}

Under the assumption  that $D$ is a bounded Lipschitz domain, it follows that $D$ is intrinsic ultracontractive (IU).  That is, for any  $\eta>0$, there is a $t_{0}=t_0(\eta, D)>0$
such that for all $t>t_0$ and all $x,y\in D$
\begin{equation}\label{IU}
\left(1-\eta\right)e^{-\lambda_{1}\left(D\right)t}\varphi_1(x)\varphi_1(y)\leq K_{D}(x,y,t)\leq\left(1+\eta\right)e^{-\lambda_{1}\left(D\right)t}\varphi_1(x)\varphi_1(y)
\end{equation}
 where $\varphi_1$ is the ground state  eigenfunction for $D$.    In fact, (IU) holds for a wider class of domains (beyond Lipschitz) and wider class of diffusion. It has been extensively studied in the literature with many different applications.   We refer the reader to \cite{Davies1989a} and \cite{Ban1991} for some of the first results on this topic that include the Lipschitz domains case.   Writing 
$$
H_D(x, y)=\int_0^{\infty}K_{D}(x,y,t) dt
$$  
for the Green's function for $D$, it follows trivially that for all IU domains $D$, $H_D(x, y)\geq C_D \varphi_1(x)\varphi_1(y)$, uniformly on $x, y\in D$.  Integrating over $D$ we see that 
\begin{equation}\label{eq-tau-varphi1}
\mathbb{E}_x\left[\tau_D\right]\geq C_D \varphi_1(x).
\end{equation} 
%If $p=1,$ then the result follows from $(\ref{Eig-Haus-1})$ together
%with $(\ref{Eig-Haus-2}).$ 
Take $\eta=1/2$. Let us first assume  $p>1$.
Applying \eqref{IU} we have  for all $x\in D$, 
\begin{align*}
\mathbb{E}_{x}\left[\tau_{D}^{p}\right] & =p\int_{0}^{\infty}t^{p-1}\mathbb{P}_{x}\left(\tau_{D}>t\right)dt\nonumber \\
 & =p\int_{0}^{t_{0}}t^{p-1}\mathbb{P}_{x}\left(\tau_{D}>t\right)dt+p\int_{t_{0}}^{\infty}t^{p-1}\int_{D}K_{D}(x,y,t/2)dydt\nonumber \\
 & \leq p\, t_{0}^{p-1} \mathbb{E}_{x}\left[\tau_{D}\right]+\frac{3}{2}\varphi_1(x)\,p\int_{t_{0}}^{\infty}t^{p-1}\int_{D}e^{-\lambda_{1}\left(D\right)t/2}\varphi_1(y)dydt\nonumber \\
 & \leq C_{1}\mathbb{E}_{x}\left[\tau_{D}\right]+C_{2}\varphi_1(x).
 \end{align*} 
 where $C_1, C_2$ are constants that depend on $p$ and $D$. Taking into account \eqref{eq-tau-varphi1} we then obtain that 
\[
\mathbb{E}_{x}\left[\tau_{D}^{p}\right]  \leq C_{p,D}\mathbb{E}_{x}\left[\tau_{D}\right]
\]
for some constant $C_{p,D}$ that only depend on $p$ and $D$.
%{\color{blue}
%On the third line of the above computations, $\int_0^{t_0}t^{p-2}dt$ doesn't seem to be integrable when $1<p<2$. Maybe we can simply use $p\int_{0}^{t_{0}}t^{p-1}\mathbb{P}_{x}\left(\tau_{D}>t\right)dt\le p t_0^{p-1} \int \mathbb{P}_x(\tau_D>t)dt= p t_0^{p-1} \mathbb{E}_x[\tau_D]$?
%}
Thus 
\begin{eqnarray}\label{IU-1}
\sup_{x\in U}\mathbb{E}_{x}\left[\left(\tau_{D}-\tau_{U}\right)^{p}\right] &\leq& C_{p,D}\sup_{x\in\partial U}\mathbb{E}_{x}\left[\tau_{D}\right]\nonumber\\
&\leq & C_{p,D}\sup_{x\in D\backslash{U}}\mathbb{E}_{x}\left[\tau_{D}\right]
\end{eqnarray}
On the other hand, for $x\in D\backslash{U}$, we have $\mathbb{P}_x(\tau_{U}>0)=0$, then 
\begin{align}\label{IU-2}
\sup_{x\in D\backslash{U}}\mathbb{E}_{x}\left[\left(\tau_{D}-\tau_{U}\right)^{p}\right] & =\sup_{x\in D\backslash{U}}\mathbb{E}_{x}\left[\tau_{D}^{p}\right]\nonumber\\
 & \leq C_{p,D}\sup_{x\in D\backslash{U}}\mathbb{E}_{x}\left[\tau_{D}\right].
 %& \leq C_{p,D}\sup_{x\in D\backslash\bar{G}}d\left(x,\partial D\right),
\end{align}
%The fact that $\mathbb{E}_x\left[\tau_D\right]\leq C_D d\left(x,\partial D\right)$ for $C^1$ domains is well known and is already mentioned and used  in \cite[Proposition 3.6.1]{Henrot-Pierre-2018}.  
Recall the fact that for a bounded Lipshitz domains, $\mathbb{E}_x\left[\tau_D\right]\leq C_D \left(d\left(x,\partial D\right)\right)^{\beta}$ where $\beta>0$ depends on the Lipszhitz character of the domain.   For the proof of the case $d=2$, which extends to any $d\geq 2$, see \cite[Proposition 2.3]{DeBlassie1987} or  the remark  in \cite[pg 199]{Banuelos-Davis1989}.  This proves the case $p\geq 1$ in (i). 

 If $0<p\leq 1$, then Jensen's inequality gives that $\mathbb{E}_{x}\left[\tau_{D}^{p}\right]\leq \left(\mathbb{E}_{x}\left[\tau_{D}\right]\right)^p$ and (ii) follows from \eqref{Eig-Haus-1}  and and \eqref{IU-2}. 
\end{proof}

%\textcolor{red}{Rodrigo:  I think when I originally wrote the proof above I has various $\bar{G}$, etc., which I think they were incorrect (what's the boundary of a closed set?). I now removed the bars.  Can you check the proof above to make sure is right?  Also, using $G$ is not good if we use $g$ for the functional, hence the change to $\Omega$.  As probabilists one should never use $\Omega$ for domains sicne it is always the sample space but what the heck, this is not really a probability paper anuhow--lol.} 
%{\color{blue}Jing: Without the bar looks good to me. I agree $\Omega$ is a little surprising, but I don't have better suggestion. Also, we have too many $G$'s and $M$'s. for instance $M_{p,d}$ and $M_p$ are two different things, $G_{p,d}$ and $G_D$ are also different things. I'm tempted to change $M_p$ to $T_p$, but am not sure.}
%\textcolor{red}{Phanuel:I think the proof is good. I changed $\Omega$ to $U$ }

\begin{prop}[Continuity of $\M_p$]
\label{prop:Existence2}For any $p>0$, the functional $\M_{p}\left(D\right)$
is continuous in the class $\mathcal{C}$ or $\mathcal{SC}$ with
respect to the Hausdorff metric. 

%\textcolor{red}{(I think we should just restrict ourselves to the case we need, namely, symmetric.  The convex case was proved in \cite{Henrot-Lucardesi-Philippin-2018} although they assume smoothness they do have the remark afterwards. I think we should avoid repeating too much what has been done because because the referee is going to 'get us" again.)}
\end{prop}

\begin{proof}
Fix $p>0$. We first prove $\M_p$ is continuous in the class $\mathcal{SC}$. Showing $\M_p$ is continuous in the class $\mathcal{C}$ is done similarly.  Let $\left\{ D_{n}\right\}\in \mathcal{SC} $ such that $D_{n}\to D\in \mathcal{SC}$ as $n\to\infty$
with respect to the Hausdorff metric. We show $\M_{p}\left(D_{n}\right)\to \M_{p}\left(D\right)$
as $n\to\infty$. 

There exists a sequence $\left\{ t_{n}\right\} \subset\mathbb{R}_{+}$
such that $t_{n}\to1$ and $t_{n}D_{n}\subset D$ for every $n$. 
By monotonicity of exit times we have for all $x\in D$ almost surely that
\begin{equation}
\tau_{t_{n}D_{n}}\leq \tau_{D}.\label{Existence_0}
\end{equation}

If $0<p<1$, using the elementary inequality $a^{p}-b^{p}\leq\left(a-b\right)^{p}$
whenever $0<b\leq a$, we have that $\mathbb{E}_{x}\left[\tau_{D}^{p}\right]\leq\mathbb{E}_{x}\left[\tau_{t_{n}D_{n}}^{p}\right]+\mathbb{E}_{x}\left[\left(\tau_{D}-\tau_{t_{n}D_{n}}\right)^{p}\right]$ for all $x\in D$. By Lemma \ref{Lem:Existence1a} (ii) and \eqref{eq-Haus} we have that
%\begin{align*}
%\mathbb{E}_{x}\left[\tau_{D}^{p}\right] %& =\mathbb{E}_{x}\left[\tau_{t_{n}D_{n}}^{p}\right]+\mathbb{E}_{x}\left[\tau_{D}^{p}-\tau_{t_{n}D_{n}}^{p}\right]\\
% & \leq\mathbb{E}_{x}\left[\tau_{t_{n}D_{n}}^{p}\right]+\mathbb{E}_{x}\left[\left(\tau_{D}-\tau_{t_{n}D_{n}}\right)^{p}\right] \\
% &\leq\mathbb{E}_{x}\left[\tau_{t_{n}D_{n}}^{p}\right]+\left(\mathbb{E}_{x}\left[\left(\tau_{D}-\tau_{t_{n}D_{n}}\right)\right]\right)^{p}\\
%\end{align*}
%Taking the supremum of both sides we obtain
%\[
%M_{p}\left(D\right)\leq M_{p}\left(t_{n}D_{n}\right)+\left(\sup_{x\in D}\mathbb{E}_{x}\left[\left(\tau_{D}-\tau_{t_{n}D_{n}}\right)\right]\right)^{p}
%\]
%
%Since $D,t_{n}D_{n}$ are either in $\mathcal{C}$ or in $\mathcal{SC}$
%then we may apply  \cite[Proposition 3.6.1]{Henrot-Pierre-2018} to $\sup_{x\in D}\mathbb{E}_{x}\left[\tau_{D}-\tau_{t_{n}D_{n}}\right]$
%to get
\begin{align}\label{Extremal_Exists_1}
\M_{p}\left(D\right)&\leq \M_{p}\left(t_{n}D_{n}\right)+\mathbb{E}_{x}\left[\left(\tau_{D}-\tau_{t_{n}D_{n}}\right)^{p}\right]\,\notag\\ %\label{Extremal_Exists_1}\\
 & \leq\mathcal{M}_{p}\left(t_{n}D\right)+C_{\beta,D}\sup_{x\in D\backslash t_{n}D}\left(d\left(x,\partial D\right)\right)^{\beta p}\notag\\
 & \leq\mathcal{M}_{p}\left(t_{n}D\right)+C_{\beta,D}\left(d_{\mathcal{H}}\left(D,t_{n}D\right)\right)^{\beta p}
\end{align}
%\textcolor{blue}{Phanuel: Here are the details using the original Henrot definition. Should we keep these details in the proof?
%\begin{align*}
% & \mathcal{M}_{p}\left(D\right)\\
% & \leq\mathcal{M}_{p}\left(t_{n}D\right)+\mathbb{E}_{x}\left[\left(\tau_{D}-\tau_{t_{n}D_{n}}\right)^{p}\right]\\
% & \leq\mathcal{M}_{p}\left(t_{n}D\right)+C_{\beta,D}\sup_{x\in D\backslash t_{n}D}\left(d\left(x,\partial D\right)\right)^{\beta p}\\
% & \leq\mathcal{M}_{p}\left(t_{n}D\right)+C_{\beta,D}\left(\max\left\{ \sup_{x\in t_{n}D\backslash D}\left(d\left(x,\partial\left(t_{n}D\right)\right)\right),\sup_{x\in D\backslash t_{n}D}\left(d\left(x,\partial D\right)\right)\right\} \right)^{\beta p}\\
% & =\mathcal{M}_{p}\left(t_{n}D\right)+C_{\beta,D}\left(d_{\mathcal{H}}\left(D,t_{n}D\right)\right)^{\beta p}
%\end{align*}
%}
where the constant $C_{\beta,D}$ depends only on $D$. 

For $p\geq1$, using the elementary inequality $x^{p}-y^{p}\leq px^{p-1}\left(x-y\right)$
whenever $0<y\leq x$,  we have that 
\begin{align}
\mathbb{E}_{x}\left[\tau_{D}^{p}-\tau_{t_{n}D_{n}}^{p}\right] & \leq p\mathbb{E}_{x}\left[\tau_{D}^{p-1}\left(\tau_{D}-\tau_{t_{n}D_{n}}\right)\right]\nonumber \\
 & \leq p\left(\mathbb{E}_{x}\left[\tau_{D}^{p}\right]\right)^{(p-1)/p}\left(\mathbb{E}_{x}\left[\left(\tau_{D}-\tau_{t_{n}D_{n}}\right)^{p}\right]\right)^{1/p}
.\label{Extremal_Exists_2}
\end{align}
Again by Lemma \ref{Lem:Existence1a} (i), we have, 

\begin{align*}
\sup_{x\in D}\mathbb{E}_{x}\left[\left(\tau_{D}-\tau_{t_{n}D_{n}}\right)^{p}\right] & \leq C_{p,D}\sup_{x\in D\backslash t_{n}D_{n}}d\left(x,\partial D\right)^{\beta}\\
 & \leq C_{p,D}\, d_{\mathcal{H}}\left(D,t_{n}D_{n}\right)^{\beta}
\end{align*}
%\textcolor{blue}{
%Phanuel: here the details for this: (Should we keep these details in the proof?)
%\begin{align*}
%\sup_{x\in D}\mathbb{E}_{x}\left[\left(\tau_{D}-\tau_{t_{n}D_{n}}\right)^{p}\right] & \leq C_{p,D}\sup_{x\in D\backslash t_{n}D}\left(d\left(x,\partial D\right)\right)^{\beta}\\
% & \leq C_{p,D}\left(\max\left\{ \sup_{x\in t_{n}D\backslash D}\left(d\left(x,\partial\left(t_{n}D\right)\right)\right),\sup_{x\in D\backslash t_{n}D}\left(d\left(x,\partial D\right)\right)\right\} \right)^{\beta}\\
% & =C_{p,D}d_{\mathcal{H}}\left(D,t_{n}D\right)^{\beta}
%\end{align*}
%}
so that 
\begin{equation}
\sup_{x\in D}\mathbb{E}_{x}\left[\tau_{D}^{p}-\tau_{t_{n}D_{n}}^{p}\right]\leq C_{p,D}^{1/p}\,p\,\M_{p}\left(D\right)^{(p-1)/p}d_{\mathcal{H}}\left(D,t_{n}D_{n}\right)^{\beta/p}.\label{Extremal_Exists_2a}
\end{equation}
Thus using \eqref{Extremal_Exists_2a} we have 
\begin{equation}\label{Extremal_Exists_2b}
\M_{p}\left(D\right)\leq \M_{p}\left(t_{n}D_{n}\right)+C_{p,D}^{1/p}\,p\,\M_{p}\left(D\right)^{(p-1)/p}d_{\mathcal{H}}\left(D,t_{n}D_{n}\right)^{\beta/p}.
\end{equation}

Together with \eqref{Extremal_Exists_1} 
we then conclude that there exist  constants $C_{p,D},C_{\beta,D}>0$ such that 
\begin{align}\label{Extremal_Exists_3}
\M_{p}\left(D\right) & \leq \M_{p}\left(t_{n}D_{n}\right)+C_{\beta,D}\left(d_{\mathcal{H}}\left(D,t_{n}D_{n}\right)\right)^{p}1_{(p<1)}\\ \nonumber
&\quad\quad +C_{p,D}^{1/p}\, p\, \M_{p}\left(D\right)^{(p-1)/p}d_{\mathcal{H}}\left(D,t_{n}D_{n}\right)^{\beta/p}1_{(p\geq1)}
\end{align}
Combining \eqref{Existence_0}, \eqref{Extremal_Exists_3} and
the fact that $\M_{p}\left(t_{n}D_{n}\right)=t_{n}^{2p}\M_{p}\left(D_{n}\right)$
gives the desired result. 
\end{proof}

\subsection{Proof of Theorem \ref{Thm:Existence} and a conjecture on the extremal}\label{Sec:5.4}
We may finally  prove our main result of this section.

\begin{proof}[Proof of Theorem \ref{Thm:Existence}]
Fix $p>0$. We consider the class of symmetric bounded convex domains $\mathcal{SC}$. The proof
is the same for $\mathcal{C}$. Let $M_{p,d}\left(\mathcal{SC}\right)=\sup_{D\in\mathcal{SC}}G_{p,d}\left(D\right)$
and pick $\{D_{n}\}\subset\mathcal{SC}$ such that 
\[
\lim_{n\to\infty}G_{p,d}\left(D_{n}\right)=M_{p,d}\left(\mathcal{SC}\right).
\]
By scaling we may assume the domains $D_{n}$ are all contained in a
fixed compact set $K$. By the Blaschke selection Theorem, there is
a subsequence $\left\{ D_{n_{k}}\right\} \subset\mathcal{SC}$ such
that $D_{n_{k}}\to D\in\mathcal{SC}$ with respect to $d_{\mathcal{H}}$.
By Lemma \ref{Lem:Existence1}, we know that $D\in SC$. We can rename
this subsequence $D_{n}$. By Equations $(3.2)$ and $(3.3)$ of \cite[page 12]{Henrot-Lucardesi-Philippin-2018}  we know that $D$ has a non-empty interior. By Proposition
\ref{prop:Existence2}, $\M_{p}$ is continuous with respect to the Hausdorff
metric in the class $\mathcal{SC}$ and $\lambda_{1}\left(D\right)$
is also well known to be continuous with respect to $d_{\mathcal{H}}$
(see \cite{Henrot-Pierre-2018}). Thus 
\[
G_{p, d}\left(D\right)=\lim_{n\to\infty}G_{p, d}\left(D_{n}\right)=M_{p,d}\left(\mathcal{SC}\right),
\]
as needed. 
\end{proof}

%\begin{remark}
%\textcolor{green}{Should we just change this into a conjecture?}{\color{blue}May be we can move the conjecture down here, so that we just give the conjecture once? }. \textcolor{red}{(I think is fine to just leave it here, no?)}  
With the existence of extremals guaranteed for all  dimension and all  $0<p<\infty$, we have the following.

\begin{conjecture}[Conjecture for $M_{p, d}\left(\mathcal{SC}\right)$]\label{Conjecture-Double-Symmetric} With the supremum taken over all domains in $\mathcal{SC}$, we have 
\[
M_{p, d}\left(\mathcal{SC}\right)= \lambda_1^p(Q_d)\E_{0}[\tau_{Q_d}^p],
\]
where 
$$
Q_{d}=\{(x_1, x_2, \dots, x_d)\in \R^d: |x_{i}|<1\},
$$ 
denotes the unit cube in $\R^d$.   
\end{conjecture}

%\textcolor{blue}{Should we remove or do anything to the following two sub-sections?}
\subsection{Remarks on conjectures; rectangles, triangles, and ellipses}

\begin{remark}[Rectangles]\label{RemRec}
Conjecture \ref{Conjecture-Double-Symmetric} in general seems to be 
nontrivial. 
In fact, even the simplest case of rectangles does not seem obvious.  More precisely, let ${\bf{a}}=(a_1, a_2, \dots, a_d)$, where $a_k>0$ for all $k$.  Set  $R_{\bf{a}}=\{x=(x_1, x_2, \dots, x_d): |x_k|<a_k, k=1, \dots, d\}$. (We call $R_{\bf{a}}$ a rectangle.) Denote  the origin in $\R^d$ by ${\bf{0}}$.   In this case we would want to show that 
%Let $I_{a_k}=(-a_k, a_k)$. We conjecture that 
%\begin{conjecture}[Rectangular domains]\label{Rectangle-vs-Square
for all  ${\bf{a}}\in\mathbb{R}^d$,
\begin{equation}\label{eq:Rectangle-vs-Square}
\lambda_{1}^p(R_{\bf{a}})\E_{\bf{0}}\left[\tau_{R_{\bf{a}}}^p\right]\leq \lambda_{1}^p(Q_d)\E_{\bf{0}}\left[\tau_{Q_d}^p\right]
\end{equation}
with equality only when $R_{\bf{a}}=Q_d$. 
%\end{conjecture} 
Since the eigenvalues of both $R_{\bf{a}}$ and $Q_d$ are explicit  and the components of the Brownian motion are independent,  the inequality \eqref{eq:Rectangle-vs-Square} can be stated in several different forms. Here is one.  Let $I_{a_k}=(-a_k, a_k)$ and recall that  $I=(-1, 1)$.  Then  \eqref{eq:Rectangle-vs-Square} is equivalent to 
\begin{equation}\label{eq:Rectangle-vs-Square-viaheat}
\left(\sum_{k=1}^d\frac{1}{a_k^2}\right)^p\int_{{0}}^{\infty} p\, t^{p-1}\prod_{k=1}^d\P_{{0}}(\tau_{I_{a_k}}>t)dt
%&=\left(\sum_{k=1}^d\frac{1}{a_k^2}\right)^p\,\int_0^{\infty}p\, t^{p-1}\prod_{k=1}^d\P_0\{\tau_{I}%>\frac{t}{a_k^2}\} dt\\
\leq d^p \int_0^{\infty}p\, t^{p-1} \left(\P_0(\tau_{I}>t)\right)^d dt.
\end{equation} 
Using the fact that $\P_{{0}}\left(\tau_{I_{a_k}}>t\right)=\P_{0}\left(\tau_{I}>\frac{t}{a_k^2}\right)$ we may even  assume that $$a_1=1<a_2<\dots<a_d.$$  

Using the fact that we know the heat kernel for an interval in terms of the eigenfunctions expansion (all which are explicitly given), the inequality has a rather appealing form.  Let us look at the case $d=2$ and $p=1$.   Then  \eqref{eq:Rectangle-vs-Square-viaheat} is equivalent to   
\begin{eqnarray}
&&\left(1+a^{2}\right)\left[1-\frac{32}{\pi^{3}}\sum_{n=0}^{\infty}\frac{\left(-1\right)^{n}}{\left(2n+1\right)^{3}}\text{sech}\left[\left(n+\frac{1}{2}\right)\frac{\pi}{a}\right]\right],\nonumber\\
&&\leq 2\left[1-\frac{32}{\pi^{3}}\sum_{n=0}^{\infty}\frac{\left(-1\right)^{n}}{\left(2n+1\right)^{3}}\text{sech}\left[\left(n+\frac{1}{2}\right)\pi\right]\right], 
\end{eqnarray} 
for all $a>1$. 

Unfortunately, despite its simplicity and all its possible formulations,  we have not been able to fully verify \eqref{eq:Rectangle-vs-Square} for all rectangles even in the case $d=2$ and $p=1$. 
\end{remark}

\begin{remark}[Triangles]\label{RemTri} 
It may be of interest to mention as well that, to the best of our knowledge, the special case of Conjecture \eqref{ConvexConj} for triangles does not seem to have been proven: 
\begin{equation}\label{trian}
\lambda_1(T)\sup_{x\in T} \E_x[\tau_T ]
 \leq \lambda_1\left(\T\right)\sup_{x\in \T}\E_{x}\left[\tau_{\T}\right],  
\end{equation}
 for all triangles $T$, where $\T$ is the equilateral triangle. Furthermore, equality holds only when $T=\T$. As pointed out in \cite[ Corollary 3.7]{Henrot-Lucardesi-Philippin-2018}, with  explicit expressions for $\E_{x}\left[\tau_{\T}\right]$ and $\lambda_1(\T)$, we have 
 $$
 \lambda_1\left(\T\right)\sup_{x\in \T}\E_{x}\left[\tau_{\T}\right]=\frac{8\pi^2}{27}\approx 2.9243. 
 $$
Combining  this with \eqref{BallCom} and \eqref{SquareCom},  we see that 
\begin{equation}\label{Ball-Square-Triangle}
\lambda_{1}\left(B\right)\sup_{x\in B}\mathbb{E}_{x}\left[\tau_{B}\right]<\lambda_{1}\left(Q_{2}\right)\sup_{x\in Q_{2}}\mathbb{E}_{x}\left[\tau_{Q_{2}}\right]<\lambda_{1}\left(\mathbb{T}\right)\sup_{x\in\mathbb{T}}\mathbb{E}_{x}\left[\tau_{\mathbb{T}}\right]. 
\end{equation}

\end{remark} 
For any convex domain $D\subset \R^2$ with finite inradius $R_D$ (supremum of radii of all disc contained in $D$), it holds that 
 \begin{equation}\label{LifeInradius} 
 \frac{1}{2}R_D^2\leq \sup_{x\in D} \E_x[\tau_D ]
 \leq \sup_{x\in S} \E_x[\tau_S ]=R_D^2, 
 \end{equation} 
 where $S\subset \R^d$ is the infinite strip of inradius $R_D$. The left hand side inequality is trivial by domain monotonicity of the exit time. For the second inequality, we refer the reader to  \cite{Sperb}. For a different proof, which extends to all moments, see \cite{Banuelos-Latala-Mendez-Hernandez-2001}. In \cite{Bartek1}, it is proved that 
 \begin{equation}\label{EigenInradius} 
 \lambda_{1}(T) R_T^2\leq \lambda_{1}(\T)R_{\T}^2=\frac{4\pi^2}{9}, 
 \end{equation}
 with equality only when $T$ is the equilateral  triangle  $\T$.  For a different proof of \eqref{EigenInradius} which uses dissymmetrization techniques, see \cite{Solynin} 
 
Although the inequalities \eqref{trian}  and  \eqref{EigenInradius} are in fact quite different and one does not imply the other, the validity of one lends credibility to the validity of the other. One can also see, for example, that with \eqref{LifeInradius} inequality  \eqref{EigenInradius} gives \eqref{trian} with a factor of 2 on the right hand side.

%\subsection{Elliptical Domains}
\begin{remark}[Ellipses]\label{RemElli} As a final remark 
we point  out that for $p=1$, both conjectures \eqref{ConvexConj} and  \ref{Conjecture-Double-Symmetric}  hold for ellipses.  In fact, the following stronger statement holds.  Let 
\[
E_{a,b}:=\left\{ \left(x,y\right)\in\mathbb{R}^{2}:\frac{x^{2}}{a^{2}}+\frac{y^{2}}{b^{2}}<1\right\} .
\]
Then, with $B$ the unit disc in $\R^2$,
\begin{equation}\label{ellip-1}
\frac{\pi^{2}}{4}\leq \lambda_1(E_{a, b})\E_{(0,0)}[\tau_{E_{(a, b)}}]\leq \lambda_{1}(B){\E_{(0,0)}}\left[\tau_{B}\right]=\frac{j_{0}^{2}}{2},
\end{equation}
To prove this inequality, it suffices to show that 
\begin{equation}
\frac{\pi^{2}}{4}\left(\frac{a^{2}+b^{2}}{a^{2}b^{2}}\right)\leq\lambda_{1}\left(E_{a,b}\right)\leq\frac{j_{0}^{2}}{2}\left(\frac{a^{2}+b^{2}}{a^{2}b^{2}}\right).\label{Eliip-Eigenvalue1}
\end{equation}
Assuming for the moment the validity of \eqref{Eliip-Eigenvalue1},  observe that since it is easy to check that 
\[
\mathbb{E}_{(x,y)}\left[\tau_{E_{a,b}}\right]=\frac{a^{2}b^{2}-b^{2}x^{2}-a^{2}y^{2}}{(a^{2}+b^{2})},
\]
by showing that the right hand side satisfies $\frac{1}{2}\Delta u=-1$ with zero boundary conditions, we have 
\begin{equation*}
\mathbb{E}_{(0,0)}\left[\tau_{E_{a,b}}\right]=\frac{a^{2}b^{2}}{a^{2}+b^{2}}
\end{equation*}
Thus the right hand side of \eqref{Eliip-Eigenvalue1} implies the right hand side of \eqref{ellip-1}.

The left hand side of \eqref{Eliip-Eigenvalue1} is trivial by domain monotonicity. Since $E_{a,b}\subset(-a,a)\times(-b,b)$, it follows immediately that 
\begin{align*}
\lambda_{1}\left(E_{a,b}\right) & \geq\lambda_{1}\left((-a,a)\times(-b,b)\right)%=\frac{\pi^{2}}{4a^{2}}+\frac{\pi^{2}}{4b^{2}}\\
  =\frac{\pi^{2}}{4}\left(\frac{a^{2}+b^{2}}{a^{2}b^{2}}\right). 
\end{align*}

%\section{appendix} 

%\begin{remark}\textcolor{red}{I am not sure what to do with this as I think it really is the same as Polya-Szego.  Take a look at the test function(eq 3, page 94).  In the case of the parabola it is the same function we use.  When we did this we were trying to see if we could also get the exit time for the rectangle in the bound. Maybe we should just quote the result.--The idea for this section of P-S comes from the so called "conformal transplantation" to use a change of variables in the eigenfunction of the unit disc as a test function.--Rodrigo}
%\end{remark}
%\bigskip

The right hand side inequality in \eqref{Eliip-Eigenvalue1} is due to Poly\'a and Szeg\"o and can be found in \cite[pg. 98]{Polya-Szego-1951}. Their proof 
%as in \cite{Polya-Szego-1951},
 is based on the technique known as conformal transplantation. To do so, one can use a  
test function $\varphi(x,y)$ with $\varphi\mid_{\partial E_{a,b}}=0$ which is an obvious modification of the eigenfunction for the disc 
and plug it into the Rayleigh quotient. Such function is given by 
$$\varphi(x,y)=J_{0}\left(j_{0}\sqrt{\frac{x^{2}}{a^{2}}+\frac{y^{2}}{b^{2}}}\right),$$
where $J_0$ is  the first Bessel function and $j_0$ is its first positive root. See \cite{Polya-Szego-1951} for details. 
%We omit the computations of the proof here. 

\end{remark}

\begin{acknowledgement}
We would like to thank Hugo Panzo for useful discussions on the topic of this paper. We would also like to thank an anonymous referee for helpful comments that helped improved the exposition of this paper. 
\end{acknowledgement}

\bibliographystyle{plain}	% (uses file "plain.bst")
%\bibliography{MainRef}

\end{document}